\documentclass[12pt]{amsart}

\usepackage{amsmath,  amssymb,  amsfonts, amsthm}
\numberwithin{equation}{section}
\usepackage{graphicx}
\usepackage{mathtools}
\usepackage{tikz}
\usetikzlibrary{intersections, calc, arrows.meta, positioning, cd, shapes.geometric}
\usepackage{hyperref}

\newtheorem{theorem}{Theorem}[section]
\newtheorem{proposition}[theorem]{Proposition}
\newtheorem{lemma}[theorem]{Lemma}
\newtheorem{corollary}[theorem]{Corollary}

\theoremstyle{definition}
\newtheorem{definition}[theorem]{Definition}
\newtheorem{example}[theorem]{Example}
\newtheorem{remark}[theorem]{Remark}

\newcommand{\Z}{\mathbb{Z}}

\newcommand{\tri}{\mathcal{T}}

\title[Conway Symbols of Frieze Patterns]{Conway Symbols of Frieze Patterns
       of Dynkin types}

\author[K. Miyamoto]{Kengo Miyamoto}
\address{Department of Computer and Information Sciences,
         Ibaraki University, Hitachi, Ibaraki 316-8511, Japan}
\email{kengo.miyamoto.uz63@vc.ibaraki.ac.jp}

\author[R. Suzuki]{Ryuhei Suzuki}
\address{Department of Computer and Information Sciences,
         Ibaraki University, Hitachi, Ibaraki 316-8511, Japan}
\email{26nm746n@vc.ibaraki.ac.jp}



\begin{document}

\begin{abstract}
In this paper, we classify the Conway symbols of the Euclidean symmetry groups of frieze patterns of Dynkin types $A_n$, $B_n$, $C_n$, $D_n$, $F_4$ and $G_2$.
First, we show that the Conway symbol of a Conway--Coxeter frieze is determined by the bilateral and the central symmetry of the underlying triangulation.
Secondly, we show that a frieze pattern of non-simply-laced type has the symbol ${*}\infty\infty$ or $\infty\infty$ according as its first interior row is palindromic or not, with the single exception of type $C_3$, which realises four symbols.
Thirdly, we show that the same dichotomy holds for type $D_n$ with $n \ge 5$, the symbol being read off the tagged triangulation of the once-punctured polygon attached to the frieze together with the value on its puncture edges, and that type $D_4$ realises five symbols.
\end{abstract}

\maketitle

\section{Introduction}
\label{sec:intro}

A \emph{frieze} is a periodic array of positive integers arranged in a horizontal band.
Friezes were introduced by Coxeter~\cite{Cox} and studied by Conway and Coxeter~\cite{CC1, CC2}, who showed that the friezes generated by the unimodular rule $ad - bc = 1$ are classified by triangulations of convex polygons.
Assem, Reutenauer and Smith~\cite{ARS} connected friezes with the cluster algebras of Fomin and Zelevinsky~\cite{FZ1, FZ2}, and the theory has since been extended in several directions, for example to friezes of types $B_n$, $C_n$ and $D_n$~\cite{BM, FoP, GES}, to $\mathrm{SL}_k$-friezes via Grassmannian cluster algebras~\cite{BFGST}, to unitary friezes and frieze vectors~\cite{GS}, and to $q$-deformations.
The $q$-deformed rational numbers of Morier-Genoud and Ovsienko~\cite{MO1} arose at the interface of continued fractions and cluster algebras, their numerators and denominators being $q$-analogues of the continuants and the $F$-polynomials of the cluster algebra of type $A$ attached to the triangulation that encodes the rational number, and they lead to the $q$-deformed Conway--Coxeter friezes of~\cite{MO2} and to the arithmetic of~\cite{KMRWY} with its applications to quivers of type $A$ and rational links.
We refer to~\cite{MG, Fa} for surveys.

To describe the symmetry of a planar pattern, Conway~\cite{C} introduced the \emph{orbifold notation}, which encodes a discrete group of Euclidean isometries by a symbol made of $*$, digits, $\times$ and $\infty$ \textup{(}see Subsection~\ref{ssec:Conway}\textup{)}.
A \emph{frieze group}, that is, a discrete group of isometries of the plane whose translation subgroup is infinite cyclic, has one of the seven Conway symbols
\[
  {*}22\infty,\quad
  2{*}\infty,\quad
  22\infty,\quad
  {*}\infty\infty,\quad
  \infty{*},\quad
  \infty{\times},\quad
  \infty\infty
\]
\textup{(}see~\cite{C, CBG}\textup{)}, whose meaning is recalled in Remark~\ref{rmk:7friezes}.
A frieze pattern of Dynkin type, embedded in the plane as in Definition~\ref{df:basic-domain}, has a frieze group of value-preserving Euclidean symmetries, and we call the Conway symbol of this group \emph{the Conway symbol} of the frieze pattern.
The aim of this paper is to determine the Conway symbols of the frieze patterns of types $A_n$, $B_n$, $C_n$, $D_n$, $F_4$ and $G_2$.

Our first aim is to classify the Conway symbols of Conway--Coxeter friezes, which we identify with frieze patterns of type $A_n$ \textup{(}Remark~\ref{rmk:equiv-FP-CC}\textup{)}.
Conway and Coxeter~\cite{CC1, CC2} observed that every Conway--Coxeter frieze admits a glide reflection and that further symmetries arise when the underlying triangulation is symmetric.
A triangulation of a convex polygon is called \emph{bilateral} or \emph{centrally symmetric} according as it is invariant under a reflection or under the half-turn about the centre of the polygon \textup{(}Subsection~\ref{ssec:main-An}\textup{)}.
The following is the first main theorem of this paper.

\begin{theorem}[Theorem~\ref{thm:main-An}]
\label{thm:intro-A}
Let $\tri$ be a triangulation of a convex $(n+3)$-gon, and let $F(q)$ be the Conway--Coxeter frieze associated with $\tri$.
Then the Conway symbol of $F(q)$ is
\[
  \begin{cases}
    2{*}\infty       & \text{if $\tri$ is bilateral and not centrally symmetric,}\\
    \infty{*}        & \text{if $\tri$ is centrally symmetric and not bilateral,}\\
    {*}22\infty      & \text{if $\tri$ is bilateral and centrally symmetric,}\\
    \infty{\times}   & \text{otherwise.}
  \end{cases}
\]
\end{theorem}

Theorem~\ref{thm:intro-A} makes the observation of Conway and Coxeter precise.

Our second aim is to treat the non-simply-laced types, in which the unimodular rule is modified on one row of the array.
The key observation is that this modification excludes every symmetry interchanging the two boundary rows, with the single exception of type $C_3$, where the modified row is the central one.
The following is the second main theorem of this paper.

\begin{theorem}[Theorem~\ref{thm:main-BCFG} and Proposition~\ref{prop:C3-classification}]
\label{thm:intro-BCFG}
Let $F$ be a frieze pattern of type $B_n$ or $C_n$ $(n \ge 2)$, other than type $C_3$, or of type $F_4$ or $G_2$.
Then the Conway symbol of $F$ is ${*}\infty\infty$ if the first interior row of $F$ is palindromic, and $\infty\infty$ otherwise.
The exceptional type $C_3$ realises exactly the four symbols $\infty\infty$, ${*}\infty\infty$, $\infty{*}$ and ${*}22\infty$.
\end{theorem}

We read Cartan matrices in the convention of Bourbaki, which interchanges $B$ and $C$ relative to~\cite{FoP, Zha} \textup{(}see the remark after Example~\ref{ex:types}\textup{)}, and in particular our $C_3$ is their $B_3$.
For type $B_n$ we refine Theorem~\ref{thm:intro-BCFG} by identifying friezes of type $B_n$ with centrally symmetric triangulations of a convex $(2n+2)$-gon, the symbol ${*}\infty\infty$ corresponding to the bilateral ones \textup{(}Proposition~\ref{prop:geometric-B}\textup{)}.

Our third aim is type $D_n$, whose Dynkin diagram has a branch.
The trivalent row of the array and one of the two leg rows are interleaved on one horizontal line, the bifurcation line, and for $n \ge 5$ this line is not central, which again excludes the symmetries interchanging the boundary rows.
Following Fontaine and Plamondon~\cite{FoP}, we attach to a frieze pattern $F$ of type $D_n$ a tagged triangulation $T_F$ of the once-punctured $n$-gon $\mathcal{P}_n$, whose $m$ puncture edges carry a common value $d$ \textup{(}Lemma~\ref{lem:FoP}\textup{)}.
The following is the third main theorem of this paper.

\begin{theorem}[Propositions~\ref{prop:D-occurrence} and~\ref{prop:D-geometric}]
\label{thm:intro-D}
Let $F$ be a frieze pattern of type $D_n$.
\begin{enumerate}
\item If $n \ge 5$, then the Conway symbol of $F$ is ${*}\infty\infty$ if $T_F$ is invariant under a reflection of $\mathcal{P}_n$ through a boundary vertex, or if $n$ is even, $d^{2} = m$ and $T_F$ is invariant under a reflection through two edge midpoints, and $\infty\infty$ otherwise.
\item Type $D_4$ realises exactly the five symbols $\infty\infty$, ${*}\infty\infty$, $\infty{*}$, $2{*}\infty$ and ${*}22\infty$.
\end{enumerate}
\end{theorem}

For odd $n$ the condition in Theorem~\ref{thm:intro-D}\,\textup{(1)} is the bilateral symmetry of $T_F$, whereas for even $n$ it can be strictly finer, as happens already for $n = 8$ \textup{(}Remark~\ref{rmk:even-n}\textup{)}.
Finally, we show that a frieze pattern of type $C_n$ is a frieze pattern of type $D_{n+1}$ whose two leg rows coincide, with one of them deleted, and that this folding preserves the Conway symbol \textup{(}Proposition~\ref{prop:fold}\textup{)}.
By Theorem~\ref{thm:intro-D}, we obtain a geometric criterion for type $C_n$ \textup{(}Corollary~\ref{cor:Cn-geometric}\textup{)}, and the exceptional types $C_3$ and $D_4$ are explained at once by the triality of $D_4$.
Table~\ref{tab:summary} at the end of the paper summarises the classification, in which the symbol $22\infty$ occurs in no type and the symbol $\infty{\times}$ occurs in type $A_n$ alone.
The types $E_6$, $E_7$ and $E_8$ lie beyond our methods, and we leave them to future work.

\section*{Acknowledgements}
The authors thank M. Wakui for introducing us to this subject.
We also thank K. Hashimoto for his Bachelor's thesis at Kansai University \textup{(}in Japanese, 2019\textup{)}.
He classified the Conway symbols in type $A_n$ for $n \le 5$, and motivated the present generalization.
This work was supported by JSPS KAKENHI Grant Number 24K16885.

\section{Preliminaries}
\label{sec:prelim}

\subsection{Frieze patterns}
\label{ssec:FP}

Let $d = (d_1, d_2, \dots, d_n)$ be an $n$-tuple of positive integers, and let $B = [b_{i,j}]$ be an $n \times n$ integer matrix satisfying
\[
  d_i b_{i,j} = - d_j b_{j,i} \quad (1 \le i, j \le n).
\]
We call $B$ a \emph{skew-symmetrizable matrix} with symmetrizer $d$.
Its \emph{generalized Cartan matrix} is $A = [a_{i,j}]$, with $a_{i,i} = 2$ and $a_{i,j} = -|b_{i,j}|$ for $i \neq j$.

\begin{definition}[{\cite{ARS, FZ1, FZ2}}]
\label{df:FP}
A \emph{frieze pattern} of the corresponding type is an array $(x_{i,j})_{1 \le i \le n,\ j \in \Z}$ of positive integers satisfying \emph{the unimodular relations}
\[
  x_{i,j}\, x_{i,j+1}
  = 1 + \prod_{k < i} x_{k,\,j+1}^{\,|a_{k,i}|}\,
        \prod_{k > i} x_{k,\,j}^{\,|a_{k,i}|}
  \qquad (1 \le i \le n,\ j \in \Z),
\]
the empty products being $1$, and the array accordingly bordered above and below by rows of $1$s.
The $j$-th \emph{diagonal} of the array is the $n$-tuple $(x_{1,j}, \dots, x_{n,j})$.
\end{definition}

\begin{remark}
\label{rmk:FP-provenance}
The unimodular relations are the exchange relations of the cluster algebra of the corresponding type, obtained by mutating successively at the indices $k$ with $b_{k,i} \le 0$ for all $i$, starting from an initial seed $(x_1, \dots, x_n)$~\cite{FZ1, FZ2, ARS}. Following~\cite{GS} we call a frieze \emph{unitary} when it takes the value $1$ on a whole cluster, the all-ones first diagonal $x_{i,1} = 1$ being the case of the initial seed.
In the finite Dynkin types treated here the array is periodic of finite width by the finite-type classification~\cite{FZ1, FZ2}.
We call the least $p > 0$ with $x_{i,j+p} = x_{i,j}$ for all $i$ and $j$ the \emph{period} of the frieze.
\end{remark}

\begin{example}
\label{ex:types}
Each of the types $A_n$, $B_n$, $C_n$, $F_4$ and $G_2$ is specified by its symmetrizer and generalized Cartan matrix $A$.
For each we then give the unitary frieze \textup{(}all-ones first diagonal\textup{)} it generates, for $n = 3$ in the three families, drawn in diamond form between rows of $1$s.

Type $A_n$ ($n \ge 1$) has symmetrizer $(1, 1, \dots, 1)$ and generalized Cartan matrix
\[
  A = \begin{pmatrix}
    2 & -1 & & & \\
    -1 & 2 & -1 & & \\
    & \ddots & \ddots & \ddots & \\
    & & -1 & 2 & -1 \\
    & & & -1 & 2
  \end{pmatrix}.
\]
For $n = 3$ the all-ones first diagonal generates the following frieze. 
\[\begin{array}{*{18}{c}}
  1 &  & 1 &  & 1 &  & 1 &  & 1 &  & 1 &  & 1 &  &  &  &  & \cdots \\
   & 1 &  & 2 &  & 2 &  & 2 &  & 1 &  & 4 &  & 1 &  &  &  & \cdots \\
   &  & 1 &  & 3 &  & 3 &  & 1 &  & 3 &  & 3 &  & 1 &  &  & \cdots \\
   &  &  & 1 &  & 4 &  & 1 &  & 2 &  & 2 &  & 2 &  & 1 &  & \cdots \\
   &  &  &  & 1 &  & 1 &  & 1 &  & 1 &  & 1 &  & 1 &  & 1 & \cdots \\
\end{array}\]

Type $B_n$ ($n \ge 2$) has symmetrizer $(1, 1, \dots, 1, 2)$ and generalized Cartan matrix
\[
  A = \begin{pmatrix}
    2 & -1 & & & \\
    -1 & 2 & -1 & & \\
    & \ddots & \ddots & \ddots & \\
    & & -1 & 2 & -2 \\
    & & & -1 & 2
  \end{pmatrix}.
\]
For $n = 3$ the all-ones first diagonal generates the following frieze.
\[\begin{array}{*{16}{c}}
  1 &  & 1 &  & 1 &  & 1 &  & 1 &  & 1 &  &  &  &  & \cdots \\
   & 1 &  & 2 &  & 2 &  & 4 &  & 1 &  & 2 &  &  &  & \cdots \\
   &  & 1 &  & 3 &  & 7 &  & 3 &  & 1 &  & 3 &  &  & \cdots \\
   &  &  & 1 &  & 10 &  & 5 &  & 2 &  & 1 &  & 10 &  & \cdots \\
   &  &  &  & 1 &  & 1 &  & 1 &  & 1 &  & 1 &  & 1 & \cdots \\
\end{array}\]

Type $C_n$ ($n \ge 2$) has symmetrizer $(2, 2, \dots, 2, 1)$ and generalized Cartan matrix
\[
  A = \begin{pmatrix}
    2 & -1 & & & \\
    -1 & 2 & -1 & & \\
    & \ddots & \ddots & \ddots & \\
    & & -1 & 2 & -1 \\
    & & & -2 & 2
  \end{pmatrix}.
\]
For $n = 3$ the all-ones first diagonal generates the following frieze.
\[\begin{array}{*{16}{c}}
  1 &  & 1 &  & 1 &  & 1 &  & 1 &  & 1 &  &  &  &  & \cdots \\
   & 1 &  & 2 &  & 2 &  & 6 &  & 1 &  & 2 &  &  &  & \cdots \\
   &  & 1 &  & 3 &  & 11 &  & 5 &  & 1 &  & 3 &  &  & \cdots \\
   &  &  & 1 &  & 4 &  & 3 &  & 2 &  & 1 &  & 4 &  & \cdots \\
   &  &  &  & 1 &  & 1 &  & 1 &  & 1 &  & 1 &  & 1 & \cdots \\
\end{array}\]

Type $F_4$ has symmetrizer $(1, 1, 2, 2)$ and generalized Cartan matrix
\[
  A = \begin{pmatrix}
    2 & -1 & 0 & 0 \\
    -1 & 2 & -2 & 0 \\
    0 & -1 & 2 & -1 \\
    0 & 0 & -1 & 2
  \end{pmatrix}.
\]
The all-ones first diagonal generates the following frieze, of period $7$.
\[\begin{array}{*{21}{c}}
1 &  & 1 &  & 1 &  & 1 &  & 1 &  & 1 &  & 1 &  & 1 &  &  &  &  &  & \cdots \\
 & 1 &  & 2 &  & 2 &  & 4 &  & 8 &  & 3 &  & 5 &  & 1 &  &  &  &  & \cdots \\
 &  & 1 &  & 3 &  & 7 &  & 31 &  & 23 &  & 14 &  & 4 &  & 1 &  &  &  & \cdots \\
 &  &  & 1 &  & 10 &  & 54 &  & 89 &  & 107 &  & 11 &  & 3 &  & 1 &  &  & \cdots \\
 &  &  &  & 1 &  & 11 &  & 5 &  & 18 &  & 6 &  & 2 &  & 2 &  & 1 &  & \cdots \\
 &  &  &  &  & 1 &  & 1 &  & 1 &  & 1 &  & 1 &  & 1 &  & 1 &  & 1 & \cdots \\
\end{array}\]

Type $G_2$ has symmetrizer $(3, 1)$ and generalized Cartan matrix
\[
  A = \begin{pmatrix} 2 & -1 \\ -3 & 2 \end{pmatrix}.
\]
The all-ones first diagonal generates the following frieze, of period $4$.
\[\begin{array}{*{13}{c}}
1 &  & 1 &  & 1 &  & 1 &  & 1 &  &  &  & \cdots \\
 & 1 &  & 2 &  & 14 &  & 9 &  & 1 &  &  & \cdots \\
 &  & 1 &  & 3 &  & 5 &  & 2 &  & 1 &  & \cdots \\
 &  &  & 1 &  & 1 &  & 1 &  & 1 &  & 1 & \cdots \\
\end{array}\]
\end{example}

We read the Cartan matrix in the convention $a_{i,j} = \langle \alpha_i, \alpha_j^{\vee} \rangle$ of Bourbaki, under which the node with the larger entry of the symmetrizer carries the short root, and the matrices of $B_n$ and $C_n$ are transposes of one another.
The cluster algebra literature \cite{FZ1, FZ2, FoP, Zha} reads the transposed matrix, and accordingly calls our $B_n$ and $C_n$ the types $C_n$ and $B_n$.

\subsection{Frieze patterns of type \texorpdfstring{$A_n$}{A\_n}}
\label{ssec:CC}

We apply Definition~\ref{df:FP} to the type $A_n$ data of Example~\ref{ex:types}.
Since $d_i = 1$ for all $i$, the matrix $B$ is skew-symmetric, with
\[
  b_{i, i+1} = -b_{i+1, i} = -1 \quad (1 \le i \le n-1),
  \qquad b_{ij} = 0 \quad \text{otherwise.}
\]
In this case the unimodular relations of Definition~\ref{df:FP} amount to the classical \emph{unimodular rule}, which requires that every four entries arranged in a diamond
\[
\begin{array}{ccc}
    & b &  \\
   a &   & d  \\
    & c &
\end{array}
\]
satisfy $ad - bc = 1$.
Explicitly, writing $x_i$ for $x_{i,j}$ and $x'_i$ for $x_{i,j+1}$ in two consecutive columns, the relations read
\[
  x_i x'_i = x'_{i-1} x_{i+1} + 1 \qquad (1 \le i \le n),
\]
under the convention $x'_0 = x_{n+1} = 1$.

Frieze patterns of type $A_n$ admit a second, classical description, due to Conway and Coxeter, which starts from a quiddity sequence.

\begin{definition}
\label{df:CC}
Let $q = (q_0, \dots, q_{N-1})$ be a sequence of positive integers, extended periodically by $q_{j+N} = q_j$.
We call $q$ a \emph{quiddity sequence} and $N$ its \emph{period}.
\emph{The Conway--Coxeter frieze} generated by $q$ is the array obtained by placing a row of $0$s, a row of $1$s offset by a half-step, the sequence $q$ below this row, and extending downward by the unimodular rule.
It is a \emph{positive frieze of width $m$} if it is bounded above and below by rows of $1$s (with rows of $0$s outside), the $m$ rows in between consisting of positive entries.
\end{definition}

By a theorem of Coxeter~\cite{Cox}, a positive frieze of width $m$ is invariant under the translation by $N = m + 3$ columns.

Quiddity sequences arise from triangulated polygons.
For a triangulation $\tri$ of a convex $N$-gon with $N \ge 4$, whose vertices $v_0, v_1, \dots, v_{N-1}$ are labelled in clockwise order, let $q_j$ denote the number of triangles of $\tri$ incident to $v_j$.
We call $q = (q_0, q_1, \dots, q_{N-1})$ \emph{the quiddity sequence} of $\tri$.

\begin{remark}
\label{rmk:equiv-FP-CC}
The frieze patterns of type $A_n$ of Definition~\ref{df:FP} and the Conway--Coxeter friezes of Definition~\ref{df:CC} are the same objects, and we use the two names interchangeably.
The unitary frieze is reproduced by extending its first interior row downward by the unimodular rule, the construction of Definition~\ref{df:CC}.
We refer to~\cite[Section~3]{ARS} for details.
\end{remark}

The next theorem of Conway and Coxeter is fundamental.

\begin{theorem}[{\cite{CC1, CC2}}]
\label{thm:CC}
For $m \ge 1$, sending a triangulation of a labelled convex $(m+3)$-gon to the frieze generated by its quiddity sequence as in \textup{Definition~\ref{df:CC}} is a bijection onto the positive friezes of width $m$.
\end{theorem}

Before stating the glide-reflection symmetry, we fix coordinates on a positive frieze and embed it in the plane, the symmetries thereby becoming Euclidean isometries.

\begin{definition}
\label{df:basic-domain}
Let $F(q)$ be a positive frieze of period $N$, generated by a quiddity sequence $q = (q_0, \dots, q_{N-1})$.
We index its rows by $i \in \{0, 1, \dots, N-2\}$ from top to bottom and its columns by $j \in \Z$, the column $j$ consisting of the entries of the $j$-th diagonal, and the $(i,j)$-entry $\zeta_{i,j}$ satisfies $\zeta_{i, j+N} = \zeta_{i, j}$.
We embed $F(q)$ in the plane by placing the entry $\zeta_{i,j}$ at the point
\[
  P(i, j) = \left( j + \dfrac{i}{2},\ -i \right) \in \mathbb{R}^2.
\]
\emph{The equator} of $F(q)$ is the horizontal line $y = -\frac{N-2}{2}$, midway between the images of the two boundary rows.
\end{definition}

In this indexing, row $0$ and row $N-2$ are the boundary rows of $1$s, rows $1, \dots, N-3$ are the interior rows, and the quiddity sequence occupies row $1$.

\begin{example}
\label{ex:CC-frieze}
Consider the triangulation of a convex pentagon with vertices $v_0, v_1, v_2, v_3, v_4$ obtained by drawing the diagonals $v_0 v_2$ and $v_0 v_3$.
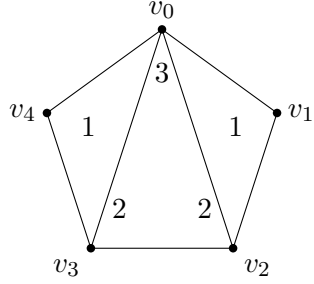
\begin{center}
\begin{tikzpicture}[scale=1.6, every node/.style={font=\small}]
  \foreach \i/\ang in {0/90, 1/18, 2/-54, 3/-126, 4/162}
    \coordinate (v\i) at (\ang:1);
  \draw (v0) -- (v1) -- (v2) -- (v3) -- (v4) -- cycle;
  \draw (v0) -- (v2);
  \draw (v0) -- (v3);
  \foreach \i in {0,...,4} \fill (v\i) circle (1pt);
  \node[above] at (v0) {$v_0$};
  \node[right] at (v1) {$v_1$};
  \node[below right] at (v2) {$v_2$};
  \node[below left] at (v3) {$v_3$};
  \node[left] at (v4) {$v_4$};
  \node at (90:0.64) {$3$};
  \node at (18:0.64) {$1$};
  \node at (-54:0.60) {$2$};
  \node at (-126:0.60) {$2$};
  \node at (162:0.64) {$1$};
\end{tikzpicture}
\end{center}
Then its quiddity sequence is $q = (3, 1, 2, 2, 1)$.
Embedded in the plane as in Definition~\ref{df:basic-domain}, $F(q)$ becomes the band below, with the equator dashed.
\begin{center}
\begin{tikzpicture}[scale=1.2, font=\small, num/.style={fill=white, inner sep=1.3pt}]
  \foreach \y in {0,-1,-2,-3}
    \draw[gray!45, {Latex[length=3.5pt]}-{Latex[length=3.5pt]}] (-0.9,\y)--(7.3,\y);
  \draw[densely dashed, thick] (-0.9,-1.5)--(7.3,-1.5);
  \node[num] at (6.85,-1.5) {\itshape equator};
  \foreach \i/\y in {0/0,1/-1,2/-2,3/-3}
    \node[gray!70, anchor=east, font=\footnotesize] at (-1.05,\y) {$i=\i$};
  \node[gray!70, font=\footnotesize] at (-0.35,0.55) {$j\to-\infty$};
  \node[gray!70, font=\footnotesize] at (6.85,0.55) {$j\to+\infty$};
  \draw[gray, dotted] (0,-0.18)--(0,-1);
  \draw[gray, -{Latex[length=3.5pt]}] (0,-1)--(0.42,-1);
  \node[gray, font=\footnotesize] at (0.22,-0.72) {$\frac12$};
  \foreach \x in {0,1,2,3,4,5} \node[num] at (\x,0) {$1$};
  \foreach \x/\v in {0.5/3,1.5/1,2.5/2,3.5/2,4.5/1,5.5/3} \node[num] at (\x,-1) {$\v$};
  \foreach \x/\v in {1/2,2/1,3/3,4/1,5/2,6/2} \node[num] at (\x,-2) {$\v$};
  \foreach \x in {1.5,2.5,3.5,4.5,5.5,6.5} \node[num] at (\x,-3) {$1$};
  \node[anchor=south] at (3.25,0.95) {$P(i,j)=\bigl(j+\frac{i}{2},\,-i\bigr)$};
\end{tikzpicture}
\end{center}
\end{example}

Crucial for our analysis is the following glide-reflection symmetry of Conway--Coxeter friezes, due to Conway and Coxeter~\cite{CC1, CC2}.

\begin{theorem}[\cite{CC1, CC2}]
\label{thm:glide}
Let $q$ be a quiddity sequence of period $N \ge 4$ and $F(q)$ the associated Conway--Coxeter frieze.
The embedded frieze is invariant under the glide reflection
\[
  g(x, y) = \left( x + \dfrac{N}{2},\ -(N-2) - y \right),
\]
the reflection about the equator followed by the horizontal translation by $\frac{N}{2}$, which in indices reads $\zeta_{i,j} = \zeta_{N-2-i,\,j+i+1}$.
\end{theorem}

The Euclidean symmetries of $F(q)$ that do not interchange the two boundary rows are governed by the symmetries of the underlying triangulation, as follows.

\begin{lemma}[cf.\ {\cite{CC1, CC2}}]
\label{lem:equivariance}
Let $\tri$ be a triangulation of a convex $N$-gon, $q = q_{\tri}$ its quiddity sequence, and $F(q)$ the associated Conway--Coxeter frieze.
\begin{enumerate}
\item A cyclic rotation $v_j \mapsto v_{j+1}$ of the $N$-gon shifts $q$ cyclically and induces the horizontal translation of $F(q)$ by one column.
\item The reflections of the $N$-gon are the maps $v_j \mapsto v_{c-j}$ with $c \in \{0, 1, \dots, N-1\}$, each reversing $q$ by replacing $q_j$ with $q_{c-j}$ and inducing the vertical reflection of $F(q)$ whose action on the quiddity row is again $j \mapsto c-j$.
\end{enumerate}
Consequently, $F(q)$ admits a vertical reflection (resp. a non-trivial horizontal translation by fewer than $N$ columns as a symmetry) if and only if $\tri$ is fixed by the corresponding reflection (resp. rotation) of the $N$-gon.
\end{lemma}

\begin{proof}
Since $F(q)$ is uniquely determined by its first row $q$ through the $a\leftrightarrow d$-invariant rule $ad-bc=1$, reversing $q$ reflects $F(q)$ in a vertical line and shifting $q$ cyclically translates it, giving (1) and (2).
A vertical reflection or horizontal translation of $F(q)$ accordingly acts on the quiddity row by a dihedral permutation $\pi$ of the vertices, and since $q_{\pi \tri} = \pi \cdot q_{\tri}$ and $q$ determines $\tri$ (Theorem~\ref{thm:CC}), it is a symmetry exactly when $\pi \tri = \tri$.
\end{proof}

\subsection{Frieze patterns of types \texorpdfstring{$B_n$, $C_n$, $F_4$ and $G_2$}{B\_n, C\_n, F\_4 and G\_2}}
\label{ssec:BCFG}

A frieze pattern of type $B_n$ or $C_n$ with $n \ge 2$, or of type $F_4$ or $G_2$, is the array of Definition~\ref{df:FP} with the data of Example~\ref{ex:types}~\cite{ARS, FoP, Zha}, its $n$ interior rows $x_1, \dots, x_n$ lying between the two boundary rows of $1$s as in~\cite[(2.1)]{Zha}, where $n = 4$ for $F_4$ and $n = 2$ for $G_2$. 
Being horizontally periodic (Remark~\ref{rmk:FP-provenance}), it has a frieze group of value-preserving Euclidean symmetries (Theorem~\ref{thm:magic}), and we embed it, together with its equator, as in Definition~\ref{df:basic-domain}.

A \emph{sub-diamond} of the embedded frieze is the set of four entries on which a single local unimodular relation is imposed, either the classical relation $ad - bc = 1$ or one of the modified relations below.
Every sub-diamond of a frieze of type $A_n$ is classical, and we record the relations of the other types one by one.

\subsubsection{Type \texorpdfstring{$B_n$}{B\_n}}
\label{sssec:Bn}

The symmetrizer is $(1, \dots, 1, 2)$, and the unimodular relations of Definition~\ref{df:FP} read
\[
  x_i x'_i = x'_{i-1} x_{i+1} + 1 \quad (1 \le i \le n-1),
  \qquad
  x_n x'_n = (x'_{n-1})^{2} + 1,
\]
with the convention $x'_0 = 1$.
The row-$x_n$ relation is the sub-diamond $ad - b^2 c = 1$ with apex squared, all others classical~\cite{Zha}.

\subsubsection{Type \texorpdfstring{$C_n$}{C\_n} ($n \ge 2$)}
\label{sssec:Cn}

The symmetrizer is $(2, \dots, 2, 1)$, and the unimodular relations of Definition~\ref{df:FP} read
\[
  x_i x'_i = x'_{i-1} x_{i+1} + 1 \quad (1 \le i \le n-2),
\]
\[
  x_{n-1} x'_{n-1} = x'_{n-2} x_n^{2} + 1,
  \qquad
  x_n x'_n = x'_{n-1} + 1,
\]
with the convention $x'_0 = 1$.
The row-$x_{n-1}$ relation is the sub-diamond $ad - bc^2 = 1$ with base squared, all others classical~\cite{Zha}.

\subsubsection{Types \texorpdfstring{$F_4$ and $G_2$}{F\_4 and G\_2}}
\label{sssec:F4G2}

For type $F_4$ the symmetrizer is $(1, 1, 2, 2)$, and the unimodular relations of Definition~\ref{df:FP} read
\[
  x_1 x'_1 = x_2 + 1, \qquad
  x_2 x'_2 = x'_1 x_3 + 1, \qquad
  x_3 x'_3 = (x'_2)^{2} x_4 + 1, \qquad
  x_4 x'_4 = x'_3 + 1.
\]
The row-$x_3$ relation is the sub-diamond $ad - b^{2} c = 1$ with apex squared, all others classical.
For type $G_2$ the symmetrizer is $(3, 1)$, and the relations read
\[
  x_1 x'_1 = x_2^{3} + 1, \qquad x_2 x'_2 = x'_1 + 1.
\]
The row-$x_1$ relation is the sub-diamond $ad - b c^{3} = 1$ with base cubed and apex $b = 1$, the other classical.

\subsection{Frieze patterns of type \texorpdfstring{$D_n$}{D\_n}}
\label{ssec:Dn}

A frieze pattern $F = (x_{i,j})$ of type $D_n$, with $n \ge 4$, is understood in the sense of Definition~\ref{df:FP}, subject to the following conventions, which we fix once and for all.
The rows carry the standard \textup{(}Bourbaki\textup{)} labelling of the Dynkin diagram, with trivalent node $n-2$ and legs $n-1$ and $n$.
The array of~\cite[(2.2)]{Zha} is the same one up to the relabelling $i \mapsto n+1-i$ of the rows, which places the legs on the rows $1$ and $2$.
The symmetrizer is $(1, \dots, 1)$, the generalized Cartan matrix is
\[
  A = \begin{pmatrix}
    2 & -1 & & & & \\
    -1 & 2 & \ddots & & & \\
    & \ddots & \ddots & -1 & & \\
    & & -1 & 2 & -1 & -1 \\
    & & & -1 & 2 & 0 \\
    & & & -1 & 0 & 2
  \end{pmatrix},
\]
and the unimodular relations of Definition~\ref{df:FP} read
\begin{equation}\label{eq:Dn-relations}
\begin{aligned}
& x_{1,j}\,x_{1,j+1} = 1 + x_{2,j}, \\
& x_{i,j}\,x_{i,j+1} = 1 + x_{i-1,j+1}\,x_{i+1,j} \quad (2 \le i \le n-3), \\
& x_{n-2,j}\,x_{n-2,j+1} = 1 + x_{n-3,j+1}\,x_{n-1,j}\,x_{n,j}, \\
& x_{n-1,j}\,x_{n-1,j+1} = 1 + x_{n-2,j+1}, \quad
  x_{n,j}\,x_{n,j+1} = 1 + x_{n-2,j+1},
\end{aligned}
\end{equation}
the second family being empty for $n = 4$.
The entries lie on the $n+1$ horizontal lines $y = 0, -1, \dots, -n$, the extreme lines being the bordering rows of $1$s.
We embed $F$ in the plane by placing $x_{i,j}$ at the point
\[
  Q(i,j) =
  \begin{cases}
    P(i,j) & (1 \le i \le n-2), \\
    \bigl(j + \frac{n-1}{2},\ -(n-2)\bigr) & (i = n-1), \\
    \bigl(j + \frac{n-1}{2},\ -(n-1)\bigr) & (i = n).
  \end{cases}
\]
Every line carries one entry per unit of horizontal period, except the line $y=-(n-2)$, which interleaves the two rows $x_{n-2}$ and $x_{n-1}$ and carries two.
We call $y=-(n-2)$ \emph{the bifurcation line}.
The equator is defined as in Definition~\ref{df:basic-domain}.

\begin{example}
\label{ex:D5-embed}
The frieze of type $D_5$ with all-ones first diagonal has period $5$, its diagonals $(x_1, \dots, x_5)$ being $(1,1,1,1,1)$, $(2,3,4,5,5)$, $(2,3,19,4,4)$, $(2,13,11,3,3)$, $(7,6,5,2,2)$. Embedded as above, we obtain the following picture.
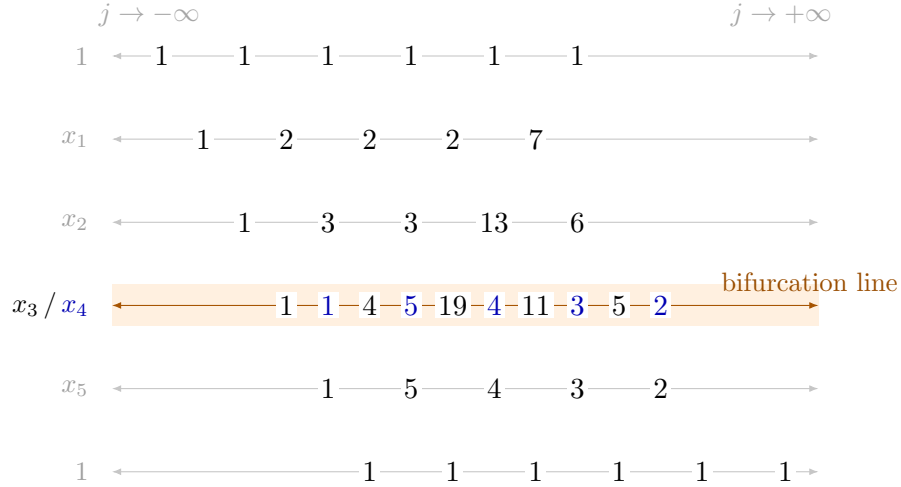
\begin{center}
\begin{tikzpicture}[scale=1.1, font=\small,
  num/.style={fill=white, inner sep=1.2pt},
  numt/.style={fill=white, inner sep=1.2pt, text=blue!70!black}]
  \fill[orange!12] (-0.6,-2.75) rectangle (7.9,-3.25);
  \foreach \y in {0,-1,-2,-4,-5}
    \draw[gray!45, {Latex[length=3.5pt]}-{Latex[length=3.5pt]}] (-0.6,\y)--(7.9,\y);
  \draw[orange!65!black, {Latex[length=3.5pt]}-{Latex[length=3.5pt]}] (-0.6,-3)--(7.9,-3);
  \node[orange!60!black, anchor=west, font=\footnotesize] at (6.6,-2.72) {bifurcation line};
  \node[gray!75, anchor=east, font=\footnotesize] at (-0.75,0) {$1$};
  \node[gray!75, anchor=east, font=\footnotesize] at (-0.75,-1) {$x_1$};
  \node[gray!75, anchor=east, font=\footnotesize] at (-0.75,-2) {$x_2$};
  \node[anchor=east, font=\footnotesize] at (-0.75,-3) {$x_3\,/\,\textcolor{blue!70!black}{x_4}$};
  \node[gray!75, anchor=east, font=\footnotesize] at (-0.75,-4) {$x_5$};
  \node[gray!75, anchor=east, font=\footnotesize] at (-0.75,-5) {$1$};
  \node[gray!70, font=\footnotesize] at (-0.15,0.5){$j\to-\infty$};
  \node[gray!70, font=\footnotesize] at (7.45,0.5){$j\to+\infty$};
  \foreach \x in {0,1,2,3,4,5} \node[num] at (\x,0) {$1$};
  \foreach \x/\v in {0.5/1,1.5/2,2.5/2,3.5/2,4.5/7} \node[num] at (\x,-1) {$\v$};
  \foreach \x/\v in {1/1,2/3,3/3,4/13,5/6} \node[num] at (\x,-2) {$\v$};
  \foreach \x/\v in {1.5/1,2.5/4,3.5/19,4.5/11,5.5/5} \node[num] at (\x,-3) {$\v$};
  \foreach \x/\v in {2/1,3/5,4/4,5/3,6/2} \node[numt] at (\x,-3) {$\v$};
  \foreach \x/\v in {2/1,3/5,4/4,5/3,6/2} \node[num] at (\x,-4) {$\v$};
  \foreach \x in {2.5,3.5,4.5,5.5,6.5,7.5} \node[num] at (\x,-5) {$1$};
\end{tikzpicture}
\end{center}
\end{example}

\subsection{Conway notation for frieze groups}
\label{ssec:Conway}

A \emph{frieze group} is a discrete group of isometries of the Euclidean plane whose translation subgroup is infinite cyclic. We recall Conway's orbifold notation~\cite{C} for the symmetries of a planar pattern, in the form needed for such groups.

\begin{definition}[{\cite[pp.~438--439]{C}}]
\label{df:Conway-features}
Let $G$ be a discrete group of isometries of the Euclidean plane.
\begin{enumerate}
\item A \emph{mirror} of $G$ is a line fixed by some reflection in $G$.
The image in $\mathbb{R}^2/G$ of a point lying on exactly $m$ mirrors is a \emph{corner-point} of order $m$.

\item A \emph{gyration} in $G$ is a rotation whose centre does \emph{not} lie on any mirror of $G$.
A point of the plane is an \emph{$m$-fold gyration point} if it is the centre of some gyration of order $m$, but not of any gyration of higher order.
Its image in $\mathbb{R}^2/G$ is a \emph{cone-point} of order $m$.

\item A \emph{glide reflection} in $G$ is a reflection composed with a non-trivial translation along its axis.
\end{enumerate}
\end{definition}

\begin{definition}[{\cite[p.~440]{C}}]
\label{df:Conway-symbol}
\emph{The Conway symbol} of a discrete group $G$ of isometries of $\mathbb{R}^2$,
equivalently of its orbifold $\mathbb{R}^2/G$, is a string of the form
\[
  \underbrace{\mathrm{o\,o\cdots o}}_{\text{handles}}\ \,
  \underbrace{A\, B\, \cdots\, C}_{\text{cone-point orders}}\ \,
  \underbrace{*a_1 b_1 \cdots c_1}_{\text{boundary curve 1}}\ \,
  \underbrace{*a_2 b_2 \cdots c_2}_{\text{boundary curve 2}}\ \,
  \cdots\ \,
  \underbrace{\times\times\cdots\times}_{\text{crosscaps}},
\]
read from left to right, with the following meaning.
\begin{itemize}
\item Each leading $\mathrm{o}$ represents a \emph{handle} of the orbifold.
\item The digits $A, B, \dots, C$ before any $*$ are the orders of the distinct cone-points.
\item Each $*$ introduces a \emph{boundary curve}, the digits $a_i, b_i, \dots, c_i$ following that $*$ being the orders of the corner-points along the curve, in consistent cyclic order.
\item Each trailing $\times$ represents a \emph{crosscap}.
\end{itemize}
In particular, a digit $m$ standing alone (not preceded by $*$) indicates an $m$-fold gyration point, whereas a substring $*m$ within a boundary curve indicates a corner-point of order $m$.
The symbol $\infty$ arises from features of infinite order coming from the translation subgroup.
\end{definition}

\begin{theorem}[{\cite[p.~446]{C}}]
\label{thm:magic}
There are exactly seven frieze groups, distinguished by the following Conway symbols.
\[
  {*}22\infty,\quad
  2{*}\infty,\quad
  22\infty,\quad
  {*}\infty\infty,\quad
  \infty{*},\quad
  \infty{\times},\quad
  \infty\infty.
\]
\end{theorem}

\begin{remark}
\label{rmk:7friezes}
Each Conway symbol admits a concrete geometric description in terms of the local features generating the group, together with Conway's descriptive name.
\begin{center}
\renewcommand{\arraystretch}{1.15}
\begin{tabular}{llp{8.5cm}}
\hline
Symbol & Name & Generators \\
\hline
$\infty\infty$       & hop            & translation only \\
$\infty\times$       & step           & glide reflection (squares to translation) \\
$\infty *$           & jump           & translation, equator mirror \\
$* \infty\infty$     & sidle          & two vertical mirrors \\
$22\infty$           & spinning hop   & two $2$-fold gyrations on the equator \\
$*22\infty$          & spinning jump  & vertical mirrors, equator mirror, $2$-fold corners \\
$2*\infty$           & spinning sidle & vertical mirrors, $2$-fold gyration on the equator \\
\hline
\end{tabular}
\end{center}

The seven types are rendered below, each as a strip of an asymmetric motif carrying the symmetry elements that generate its group.

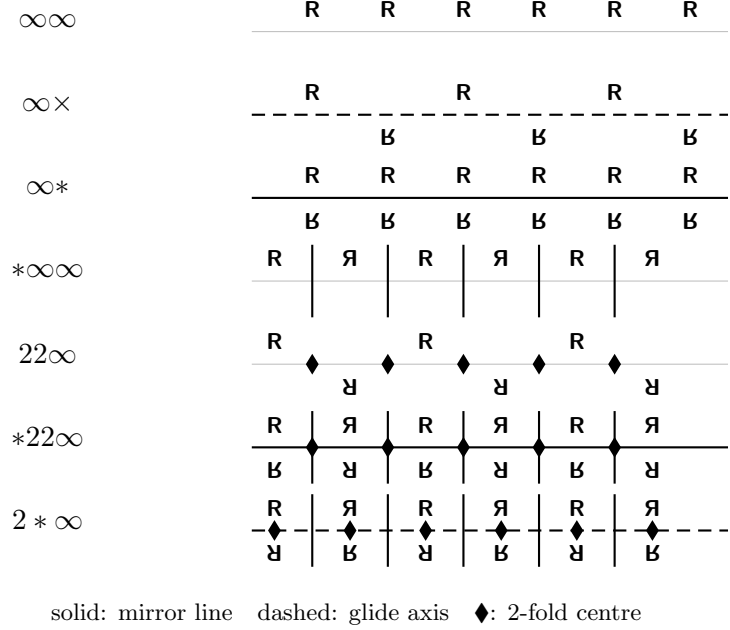
\begin{center}
\begin{tikzpicture}[font=\small,
  mot/.style={font=\sffamily\bfseries, inner sep=0pt, scale=0.72},
  eq/.style={gray!50},
  mir/.style={thick},
  gli/.style={thick, dash pattern=on 5pt off 3pt},
  twofold/.style={diamond, fill=black, inner sep=0pt, minimum width=5pt, minimum height=8pt},
  ]
\def\u{0.30}\def\l{-0.30}
\def\R#1#2#3{\node[mot,#3] at (#1,#2) {R};}
\def\rt#1{\node[twofold] at (#1,0) {};}
\def\vmir#1{\draw[mir] (#1,-0.48)--(#1,0.48);}
\begin{scope}[yshift=0cm]
  \node at (-2.5,0.14) {$\infty\infty$};
  \draw[eq] (0.2,0)--(6.5,0);
  \foreach \x in {1,2,3,4,5,6} {\R{\x}{\u}{}}
\end{scope}
\begin{scope}[yshift=-1.1cm]
  \node at (-2.5,0.14) {$\infty\times$};
  \draw[gli] (0.2,0)--(6.5,0);
  \foreach \x in {1,3,5} {\R{\x}{\u}{}}
  \foreach \x in {2,4,6} {\R{\x}{\l}{yscale=-1}}
\end{scope}
\begin{scope}[yshift=-2.2cm]
  \node at (-2.5,0.14) {$\infty*$};
  \draw[mir] (0.2,0)--(6.5,0);
  \foreach \x in {1,2,3,4,5,6} {\R{\x}{\u}{}\R{\x}{\l}{yscale=-1}}
\end{scope}
\begin{scope}[yshift=-3.3cm]
  \node at (-2.5,0.14) {$*\infty\infty$};
  \draw[eq] (0.2,0)--(6.5,0);
  \foreach \x in {1,2,3,4,5} {\vmir{\x}}
  \foreach \x in {0.5,2.5,4.5} {\R{\x}{\u}{}}
  \foreach \x in {1.5,3.5,5.5} {\R{\x}{\u}{xscale=-1}}
\end{scope}
\begin{scope}[yshift=-4.4cm]
  \node at (-2.5,0.14) {$22\infty$};
  \draw[eq] (0.2,0)--(6.5,0);
  \foreach \x in {1,2,3,4,5} {\rt{\x}}
  \foreach \x in {0.5,2.5,4.5} {\R{\x}{\u}{}}
  \foreach \x in {1.5,3.5,5.5} {\R{\x}{\l}{rotate=180}}
\end{scope}
\begin{scope}[yshift=-6.6cm]
  \node at (-2.5,0.14) {$2*\infty$};
  \draw[gli] (0.2,0)--(6.5,0);
  \foreach \x in {1,2,3,4,5} {\vmir{\x}}
  \foreach \x in {0.5,1.5,2.5,3.5,4.5,5.5} {\rt{\x}}
  \foreach \x in {0.5,2.5,4.5} {\R{\x}{\u}{}\R{\x}{\l}{rotate=180}}
  \foreach \x in {1.5,3.5,5.5} {\R{\x}{\u}{xscale=-1}\R{\x}{\l}{yscale=-1}}
\end{scope}
\begin{scope}[yshift=-5.5cm]
  \node at (-2.5,0.14) {$*22\infty$};
  \draw[mir] (0.2,0)--(6.5,0);
  \foreach \x in {1,2,3,4,5} {\vmir{\x}}
  \foreach \x in {1,2,3,4,5} {\rt{\x}}
  \foreach \x in {0.5,2.5,4.5} {\R{\x}{\u}{}\R{\x}{\l}{yscale=-1}}
  \foreach \x in {1.5,3.5,5.5} {\R{\x}{\u}{xscale=-1}\R{\x}{\l}{rotate=180}}
\end{scope}
\node[anchor=west,font=\scriptsize] at (-2.6,-7.7) {solid: mirror line\quad dashed: glide axis\quad $\blacklozenge$: $2$-fold centre};
\end{tikzpicture}
\end{center}
\end{remark}

The following identification criterion, a direct reading of the generator description in Remark~\ref{rmk:7friezes}, is what we use to pin down the Conway symbol of a given frieze.

\begin{lemma}
\label{lem:identification}
Let $G$ be a frieze group containing the horizontal glide reflection of \textup{Theorem~\ref{thm:glide}}. Then $G$ is one of the four types $\infty{\times}$, $\infty{*}$, $2{*}\infty$, ${*}22\infty$. Moreover, we have the following.
\begin{itemize}
\item[(a)] $G$ contains a vertical reflection if and only if $G$ is of type $2{*}\infty$ or ${*}22\infty$, and
\item[(b)] $G$ contains an equator reflection if and only if $G$ is of type $\infty{*}$ or ${*}22\infty$.
\end{itemize}
\end{lemma}

\begin{proof}
Among the seven Conway symbols of \textup{Theorem~\ref{thm:magic}}, the types $\infty\infty$, ${*}\infty\infty$ and $22\infty$ contain no horizontal glide, whence $G$ is one of the remaining four $\infty{\times}$, $\infty{*}$, $2{*}\infty$, ${*}22\infty$. Reading off the vertical and equatorial reflections of these four from the same table yields (a) and (b).
\end{proof}

\subsection{Boundary-exchange symmetry}
\label{ssec:boundary-exchange}

The classification of types $B_n$, $C_n$, $D_n$, $F_4$ and $G_2$ turns on whether a frieze admits a symmetry interchanging its two boundary rows, and we isolate the notion and its consequences here.

\begin{definition}
\label{df:boundary-exchange}
A frieze pattern $F$ is said to admit \emph{the boundary-exchange symmetry} if the Euclidean symmetry group of $F$ contains an isometry that interchanges its top and bottom boundary rows.
\end{definition}

A boundary-exchange reverses the vertical direction of the band.

\begin{lemma}
\label{lem:boundary-mech}
A boundary-exchange of a frieze of type $A_n$, $B_n$, $C_n$, $F_4$ or $G_2$, with interior rows $x_1, \dots, x_n$, carries the row $x_i$ to the row $x_{n+1-i}$, and each sub-diamond to a sub-diamond on the image rows, exchanging apex and base and preserving the unordered pair of side entries.
\end{lemma}

\begin{proof}
A boundary-exchange is a value-preserving isometry reversing the vertical direction of the band, which on the interior rows reads $x_i \mapsto x_{n+1-i}$, and it therefore carries each sub-diamond to a sub-diamond with apex and base interchanged and with the side vertices carried to the side vertices.
\end{proof}

\begin{lemma}
\label{lem:boundary-exch-list}
Among the seven Conway symbols of \textup{Theorem~\ref{thm:magic}}, the frieze groups that admit boundary-exchange symmetry are exactly those whose Conway symbol belongs to
\[
  \{\, \infty{*}, \ \infty{\times}, \ 22\infty, \ 2{*}\infty, \
    {*}22\infty \,\}.
\]
\end{lemma}

\begin{proof}
By the above, a boundary-exchange is a reflection about the equator, a glide along it, or a $2$-fold rotation centred on the equator, these being the band isometries that reverse the vertical direction.
Consulting Remark~\ref{rmk:7friezes}, the listed symbols are exactly those with one of them among their generators.
\end{proof}

\section{Classification for types \texorpdfstring{$A_n$, $B_n$, $C_n$, $F_4$ and $G_2$}{A\_n, B\_n, C\_n, F\_4 and G\_2}}
\label{sec:main-ABCFG}

\subsection{Classification of friezes of type \texorpdfstring{$A_n$}{A\_n}}
\label{ssec:main-An}

We say that a triangulation $\tri$ of a convex polygon is \emph{bilateral} if it admits a line of symmetry, and \emph{centrally symmetric} if it is invariant under a $180^\circ$ rotation about the centre of the polygon.

By Lemma~\ref{lem:equivariance}, $\tri$ is bilateral if and only if $F(q)$ admits a vertical reflection.
The central symmetry is detected as follows.

\begin{proposition}
\label{prop:reflections}
Let $\tri$ be a triangulation of a convex $N$-gon, $q = q_{\tri}$ its quiddity sequence, and $F(q)$ the associated Conway--Coxeter frieze.
Then $\tri$ is centrally symmetric if and only if the equator of $F(q)$ is a reflection line.
\end{proposition}

\begin{proof}
Suppose first that $\tri$ is centrally symmetric.
Then its $N$-gon has an even number of vertices, $N = 2\ell$, and $q$ has period $\ell$ (not necessarily minimal).
Hence the horizontal translation $t$ by $\ell = \frac{N}{2}$ columns is a symmetry of $F(q)$.
Writing $s$ for the reflection about the equator, the glide reflection of Theorem~\ref{thm:glide} factors as $g = t \circ s$.
Since $t$ and $g$ are symmetries, $s = t^{-1} \circ g$ is one also, and the equator is a reflection line.

Conversely, suppose the equator reflection $s$ is a symmetry of $F(q)$.
Then $t = g \circ s$ is a horizontal translation by $\frac{N}{2}$.
Being a symmetry, $t$ preserves the rows and carries integer columns to integer columns.
Hence $\frac{N}{2}$ is an integer, $N$ is even, and $q$ has period $\frac{N}{2}$.
By Lemma~\ref{lem:equivariance}, the $180^\circ$ rotation of the $N$-gon fixes $\tri$, that is, $\tri$ is centrally symmetric.
\end{proof}

\begin{theorem}
\label{thm:main-An}
Let $\tri$ be a triangulation of a convex $(n+3)$-gon, and let $F(q)$ be the Conway--Coxeter frieze generated by its quiddity sequence $q$.
Then the Conway symbol of $F(q)$ is
\[
  \begin{cases}
    2{*}\infty       & \text{if $\tri$ is bilateral and not centrally symmetric,}\\
    \infty{*}        & \text{if $\tri$ is centrally symmetric and not bilateral,}\\
    {*}22\infty      & \text{if $\tri$ is bilateral and centrally symmetric,}\\
    \infty{\times}   & \text{otherwise.}
  \end{cases}
\]
\end{theorem}

\begin{proof}
The frieze $F(q)$ always admits the glide reflection of Theorem~\ref{thm:glide}.
By Lemma~\ref{lem:identification} its Conway symbol is one of $\infty{\times}$, $\infty{*}$, $2{*}\infty$, ${*}22\infty$, determined by whether $F(q)$ admits a vertical reflection $(a)$ and an equator reflection $(b)$.
By Lemma~\ref{lem:equivariance}, $(a)$ holds if and only if $\tri$ is bilateral.
By Proposition~\ref{prop:reflections}, $(b)$ holds if and only if $\tri$ is centrally symmetric.
The correspondence of Lemma~\ref{lem:identification} then yields the symbol in each of the four cases.
\end{proof}

\subsection{Classification of friezes of types \texorpdfstring{$B_n$, $C_n$, $F_4$ and $G_2$}{B\_n, C\_n, F\_4 and G\_2}}
\label{ssec:main-BCFG}

\begin{lemma}
\label{lem:no-boundary-exchange}
A frieze pattern of type $B_n$ $(n \ge 2)$, of type $C_n$ $(n \ge 2,\ n \ne 3)$, of type $F_4$ or of type $G_2$ does not admit boundary-exchange symmetry.
\end{lemma}

\begin{proof}
Write $m$ for the modified centre row, $m = n$ for $B_n$, $m = n-1$ for $C_n$, $m = 3$ for $F_4$ and $m = 1$ for $G_2$ (Subsection~\ref{ssec:BCFG}), and $e$ for the exponent in the modified relation, $e = 3$ for $G_2$ and $e = 2$ otherwise.
By Lemma~\ref{lem:boundary-mech} a boundary-exchange $\sigma$ carries the modified sub-diamonds, centred on row $x_m$, to sub-diamonds centred on row $x_{n+1-m}$, and $n+1-m$ equals $1$ for type $B_n$ and $2$ for the other three types.
As $n \ge 2$ for type $B_n$ and $n \ne 3$ for type $C_n$, the image row differs from the modified row in every case, and the image sub-diamonds obey the classical relation.
The vertex raised to the power $e$ in a modified sub-diamond ranges over row $x_{n-1}$ for $B_n$, over row $x_n$ for $C_n$ and over row $x_2$ for $F_4$ and $G_2$, and that row is not identically $1$.
Indeed, the relation centred on that row, namely
\[ x_{n-1,j}\, x_{n-1,j+1} = x_{n-2,j+1}\, x_{n,j} + 1 \]
for $B_n$, $x_{n,j}\, x_{n,j+1} = x_{n-1,j+1} + 1$ for $C_n$, $x_{2,j}\, x_{2,j+1} = x_{1,j+1}\, x_{3,j} + 1$ for $F_4$ and $x_{2,j}\, x_{2,j+1} = x_{1,j+1} + 1$ for $G_2$, makes the product of two consecutive entries at least $2$.
Choose a modified sub-diamond $\mathcal{D}$ whose raised vertex $v$ is at least $2$, and write $w$ for the remaining apex-or-base vertex.
The relation at $\mathcal{D}$ reads $ad - v^{e} w = 1$.
By Lemma~\ref{lem:boundary-mech} the image $\sigma(\mathcal{D})$ shares the side product $ad$ and the values $v, w$, with apex and base interchanged, and obeys the classical relation, giving $ad - vw = 1$.
Subtracting gives $v^{e} w = v w$, whence $v = 1$, a contradiction.
\end{proof}

\begin{remark}
\label{rmk:C3-counterexample}
Lemma~\ref{lem:no-boundary-exchange} fails for $C_3$.
A counterexample is the constant frieze with $x_1 = x_3 = 2$ and $x_2 = 3$,
\[
\begin{array}{*{11}{c}}
  & 1 & & 1 & & 1 & & 1 & & 1 & \\
2 & & 2 & & 2 & & 2 & & 2 & & \cdots \\
  & 3 & & 3 & & 3 & & 3 & & 3 & \\
2 & & 2 & & 2 & & 2 & & 2 & & \cdots \\
  & 1 & & 1 & & 1 & & 1 & & 1 &
\end{array}
\]
which obeys the type-$C_3$ relations ($2 \cdot 2 - 1 \cdot 3 = 1$ on $x_1, x_3$ and $3 \cdot 3 - 2 \cdot 2^{2} = 1$ on $x_2$) and is invariant under its equator reflection, a boundary-exchange.
Its Conway symbol is ${*}22\infty$, computed in Proposition~\ref{prop:C3-classification}.
\end{remark}

\begin{lemma}
\label{lem:x1-determines}
Let $F$ be a frieze pattern of type $B_n$ or $C_n$ $(n \ge 2)$, or of type $F_4$ or $G_2$, with first interior row $x_1 = (x_{1,j})_{j \in \Z}$.
\begin{enumerate}
\item $F$ is determined by $x_1$.
\item $F$ admits a vertical reflection if and only if $x_1$ is \emph{palindromic}, that is, $x_{1,\,c-j} = x_{1,j}$ for some $c \in \Z$ and all $j$.
\end{enumerate}
\end{lemma}

\begin{proof}
(1) The sub-diamond centred on row $x_i$ has its apex in row $x_{i-1}$, its base in row $x_{i+1}$, and its two side vertices in row $x_i$.
Solving the local relation for the base expresses each entry of row $x_{i+1}$ through entries of rows $x_{i-1}$ and $x_i$.
A classical relation gives the base as $\frac{ad-1}{b}$, and the modified relation of type $C_n$, with squared base, gives it as the positive square root of $\frac{ad-1}{b}$.
The modified relation of type $F_4$, with squared apex, gives the base as $\frac{ad-1}{b^{2}}$, and that of type $G_2$, with cubed base and apex $1$, gives it as the positive cube root of $ad - 1$.
For type $B_n$ the descent uses the classical relations alone.
Starting from the boundary row $x_0 \equiv 1$ and the given row $x_1$, this determines $x_2, \dots, x_n$ in turn.
The one local relation not consumed by the descent, namely the relation centred on row $x_n$, is then a consistency condition, satisfied because $F$ is a frieze.
Hence $x_1$ determines $F$.
This argument establishes uniqueness rather than existence, and we do not claim that an arbitrary positive periodic row is the first interior row of a frieze of the given type.

(2) A vertical reflection of the embedded frieze is the reflection $r$ in a vertical line $x = \frac{c}{2}$ with $c \in \Z$.
It preserves each row and acts on row $x_i$ by $j \mapsto c_i - j$ with $c_i = c - i$.
Because the classical and modified relations are symmetric in the two side vertices $a \leftrightarrow d$, the reflected array $r(F)$ again satisfies the relations of the same type and is bounded by the same rows of $1$s.
Hence $r(F)$ is a frieze of the same type, and its first interior row is the reversal $(x_{1,\,c_1-j})_{j}$ of $x_1$.
By part~(1), two friezes of the same type with the same first interior row coincide, and thus $r(F) = F$ exactly when $x_{1,\,c_1-j} = x_{1,j}$ for all $j$.
Since $c_1$ ranges over $\Z$ as the axis varies, $F$ admits a vertical reflection if and only if $x_1$ is palindromic.
\end{proof}

\begin{theorem}
\label{thm:main-BCFG}
Let $F$ be a frieze pattern of type $B_n$ or $C_n$ with $n \ge 2$, other than type $C_3$, or of type $F_4$ or $G_2$, and let $x_1$ be its first interior row.
Then the Conway symbol of $F$ is ${*}\infty\infty$ if $x_1$ is palindromic, and $\infty\infty$ otherwise.
\end{theorem}

\begin{proof}
By Lemma~\ref{lem:no-boundary-exchange} the frieze admits no boundary-exchange symmetry, and by Lemma~\ref{lem:boundary-exch-list} its Conway symbol is $\infty\infty$ or ${*}\infty\infty$.
These two are distinguished by the presence of a vertical reflection.
By Lemma~\ref{lem:x1-determines}\,(2) a vertical reflection is present exactly when $x_1$ is palindromic, which yields the stated dichotomy.
\end{proof}

Theorem~\ref{thm:main-BCFG} is the analogue of Theorem~\ref{thm:main-An} for the non-simply-laced types, the first interior row $x_1$ playing the role that the quiddity sequence plays for type $A_n$.
For $G_2$ the constant frieze with $x_1 \equiv 3$ and $x_2 \equiv 2$ has symbol ${*}\infty\infty$, while the unitary frieze of Example~\ref{ex:types}, whose first interior row $(1, 2, 14, 9)$ is not palindromic, has symbol $\infty\infty$.
For $F_4$ the first diagonal $(1, 2, 3, 2)$ generates a frieze of period $7$ whose first interior row $(1, 3, 2, 4, 4, 2, 3)$ is palindromic, and its symbol is ${*}\infty\infty$.

For type $B_n$ this dichotomy has a transparent geometric source, the folding of centrally symmetric triangulations.

The bijection of the following proposition is due to Fontaine and Plamondon~\cite[Theorem~4.2]{FoP}, who obtain it by folding the cluster algebra of type $A_{2n-1}$, the folded type being their $C_n$ (our $B_n$).
The proof below gives an elementary, self-contained account at the level of frieze arrays, and reads off the Conway symbol.

\begin{proposition}
\label{prop:geometric-B}
Folding induces a bijection between the centrally symmetric triangulations of a convex $(2n+2)$-gon and the frieze patterns of type $B_n$, under which the first interior row of the frieze is the quiddity sequence of the triangulation.
The Conway symbol of the frieze is ${*}\infty\infty$ if the triangulation is bilateral, and $\infty\infty$ otherwise.
\end{proposition}

\begin{proof}
Let $\tri$ be a centrally symmetric triangulation of the $(2n+2)$-gon and let $F$ be its Conway--Coxeter frieze, of width $2n-1$, with interior rows $z_1,\dots,z_{2n-1}$ and quiddity $z_1=q$.
By Proposition~\ref{prop:reflections} the central symmetry of $\tri$ is the equator reflection $s$ of $F$.
It fixes the middle row $z_n$ and, on the embedded array $P(i,j)=(j+\frac i2,-i)$, sends $(i,j)$ to $(2n-i,\,j+i-n)$, and therefore $z_{n+1,\,j}=z_{n-1,\,j+1}$ for every $j$.
Put $x_i=z_i$ for $1\le i\le n$.
For $1\le i\le n-1$ the sub-diamond of $F$ centred on $z_i$ is classical and involves only $z_{i-1},z_i,z_{i+1}$ (with $z_0\equiv1$), and thus the $x_i$ obey the classical relation.
The sub-diamond centred on $z_n$ is classical, namely $z_{n,j}z_{n,j+1}-z_{n-1,j+1}\,z_{n+1,j}=1$, and substituting $z_{n+1,j}=z_{n-1,j+1}$ turns it into
\[
  x_{n,j}\,x_{n,j+1}-x_{n-1,j+1}^{\,2}=1,
\]
the modified relation $ad-b^{2}c=1$ of type $B_n$ with boundary value $c=1$.
Hence $x_1,\dots,x_n$ is a frieze of type $B_n$ with first interior row $x_1=q$.
To see that folding is a bijection we exhibit its inverse.
Let $(x_{i,j})_{1\le i\le n}$ be an arbitrary frieze of type $B_n$, with boundary rows $x_0\equiv x_{n+1}\equiv1$, and define an array $(z_{i,j})_{0\le i\le 2n,\ j\in\Z}$ by
\[
  z_{i,j}=x_{i,j}\quad(0\le i\le n),
  \qquad
  z_{n+k,\,j}=x_{n-k,\,j+k}\quad(1\le k\le n).
\]
Then $z_0\equiv z_{2n}\equiv1$, and by construction $z$ is invariant under the reflection $s$ in the line $y=-n$, which sends $(i,j)$ to $(2n-i,\,j+i-n)$.
We check that every sub-diamond of $z$ satisfies the classical relation.
Those centred on a row $z_i$ with $1\le i\le n-1$ are the classical sub-diamonds of $x$.
For the one centred on the middle row $z_n$ the same computation follows in reverse from the modified relation satisfied by $x$, since $z_{n+1,j}=x_{n-1,j+1}$.
Those centred on a lower row $z_{n+k}$ $(1\le k\le n-1)$ are the $s$-images of the classical sub-diamonds centred on $z_{n-k}=x_{n-k}$, and therefore again classical, the reflection $s$ interchanging apex and base and leaving $ad-bc$ invariant.
Thus $z$ is a positive frieze of width $2n-1$ bounded by rows of $1$s, and by Theorem~\ref{thm:CC} it is the Conway--Coxeter frieze of a triangulation $\tri$ of a convex $(2n+2)$-gon, with quiddity $z_1=x_1$.
By Proposition~\ref{prop:reflections} its $s$-invariance makes $\tri$ centrally symmetric.
Folding records the upper $n$ interior rows of this $s$-symmetric array, and the construction above reconstitutes the array from them, and thus the two assignments are mutually inverse.
Finally, $q$ is palindromic if and only if $\tri$ is bilateral (Lemma~\ref{lem:equivariance}), while the symbol is ${*}\infty\infty$ if and only if $x_1=q$ is palindromic (Theorem~\ref{thm:main-BCFG}).
The criterion follows.
\end{proof}

Type $C_n$ is the folding of $D_{n+1}$ rather than of $A_{2n-1}$, and not all of its friezes are unitary, the constant frieze of Remark~\ref{rmk:C3-counterexample} being an example.
Its counterpart of Proposition~\ref{prop:geometric-B} is Proposition~\ref{prop:fold} together with Corollary~\ref{cor:Cn-geometric}.

The exceptional type $C_3$ realises four symbols, which the following proposition determines.

\begin{proposition}
\label{prop:C3-classification}
Let $F$ be a frieze pattern of type $C_3$, with interior rows $x_1, x_2, x_3$.
\begin{enumerate}
\item $F$ admits a boundary-exchange symmetry if and only if $x_{3,j} = x_{1,j+1}$ for all $j$, equivalently, the apex equals the base in every modified sub-diamond.
Whenever $F$ admits a boundary-exchange, the reflection in the equator is itself a symmetry of $F$.
\item Accordingly the Conway symbol of $F$ is determined by the two independent invariants ``$x_1$ palindromic'' and ``$F$ boundary-exchange-symmetric'' as in the following table.
\[
\renewcommand{\arraystretch}{1.15}
\begin{array}{cc|c}
 \text{$x_1$ palindromic} & \text{boundary-exchange} & \text{symbol}\\ \hline
 \text{no}  & \text{no}  & \infty\infty\\
 \text{yes} & \text{no}  & {*}\infty\infty\\
 \text{no}  & \text{yes} & \infty{*}\\
 \text{yes} & \text{yes} & {*}22\infty
\end{array}
\]
\end{enumerate}
\end{proposition}

\begin{proof}
(1) Assume first that $x_{3,j} = x_{1,j+1}$ for all $j$.
The reflection $s$ in the equator fixes the middle row $x_2$ pointwise and carries $P(1,j)$ to $P(3,j-1)$.
The values agree precisely when $x_{1,j} = x_{3,j-1}$, which is the hypothesis.
Hence $s$ is a symmetry interchanging the two boundary rows, a boundary-exchange.
Conversely, let $\sigma$ be any boundary-exchange.
By Lemma~\ref{lem:boundary-mech} it fixes the middle row $x_2$ and carries each modified sub-diamond $\mathcal{D}$ to a modified sub-diamond $\sigma(\mathcal{D})$, with apex and base exchanged.
Write $w$ for the apex of $\mathcal{D}$ (in $x_1$) and $v$ for its squared base (in $x_3$).
The modified relations at $\mathcal{D}$ and $\sigma(\mathcal{D})$ read $ad - w v^{2} = 1$ and $ad - v w^{2} = 1$, with common side product $ad$, and subtracting them gives $v w (v-w) = 0$, whence $v = w$.
Thus apex equals base at every modified sub-diamond, that is $x_{3,j} = x_{1,j+1}$ for all $j$.
By the forward direction the equator reflection then lies in the symmetry group of $F$, which is the second assertion.

(2) Beyond translation, two features may occur, a vertical reflection $r$ governed by Lemma~\ref{lem:x1-determines}\,(2) and a boundary-exchange, the latter present by part~(1) exactly when the equator reflection $s$ is a symmetry.
Suppose first that $F$ admits no boundary-exchange.
By Lemma~\ref{lem:boundary-exch-list} its symbol is then $\infty\infty$ or ${*}\infty\infty$, the former being translations alone and the latter adjoining the vertical mirrors, and $r$ separates the two.
Suppose instead that $F$ admits a boundary-exchange.
Then $s$ lies in its symmetry group, and among the seven symbols only $\infty{*}$ and ${*}22\infty$ carry an equator mirror (Remark~\ref{rmk:7friezes}).
Here too $r$ separates them, the vertical mirrors of ${*}22\infty$ meeting the equator mirror in $2$-fold corner points.
This gives the four entries of the table.
\end{proof}

\begin{example}
\label{ex:C3-symbols}
All four symbols of \textup{Proposition~\ref{prop:C3-classification}} occur.
A triple $(a_1,a_2,a_3)$ denotes the first diagonal $x_{i,1}=a_i$.
The displayed friezes have period $4$ and are drawn over one period in the diamond layout of Example~\ref{ex:CC-frieze}, the top and bottom rows being the boundary $1$s.
The unitary frieze $(1,1,1)$, displayed in Example~\ref{ex:types}, is neither palindromic nor boundary-exchange-symmetric and gives $\infty\infty$.
The first diagonal $(1,2,1)$ generates
\[
\begin{array}{*{11}{c}}
  & 1 & & 1 & & 1 & & 1 & & 1 & \\
1 & & 3 & & 1 & & 6 & & 1 & & \cdots \\
  & 2 & & 2 & & 5 & & 5 & & 2 & \\
3 & & 1 & & 3 & & 2 & & 3 & & \cdots \\
  & 1 & & 1 & & 1 & & 1 & & 1 &
\end{array}
\]
in which $x_1$ is palindromic and no boundary-exchange occurs, giving ${*}\infty\infty$.
The first diagonal $(1,1,2)$ generates
\[
\begin{array}{*{11}{c}}
  & 1 & & 1 & & 1 & & 1 & & 1 & \\
1 & & 2 & & 5 & & 3 & & 1 & & \cdots \\
  & 1 & & 9 & & 14 & & 2 & & 1 & \\
1 & & 2 & & 5 & & 3 & & 1 & & \cdots \\
  & 1 & & 1 & & 1 & & 1 & & 1 &
\end{array}
\]
in which $x_{3,j}=x_{1,j+1}$ for all $j$ while $x_1$ is not palindromic, giving $\infty{*}$.
Finally the constant frieze $(2,3,2)$ of Remark~\ref{rmk:C3-counterexample}, with $x_1$ palindromic and a boundary-exchange, gives ${*}22\infty$.
\end{example}

\section{Classification for type \texorpdfstring{$D_n$}{D\_n}}
\label{sec:main-D}

Type $D_n$ departs from the previous types.
Its array carries the doubly-dense bifurcation line of Subsection~\ref{ssec:Dn}, and the vertical reflection must be read from the once-punctured polygon of Schiffler and Baur--Marsh rather than from a convex polygon.

\subsection{The array-level dichotomy}
\label{ssec:main-Dn}

We now classify the Conway symbols of friezes of type $D_n$.
As for the other types, the symbol is determined by two features beyond translation, a boundary-exchange and a vertical reflection.
We take them in turn.

\begin{lemma}\label{lem:D-no-boundary}
For $n \ge 5$, a frieze pattern of type $D_n$ admits no boundary-exchange symmetry.
\end{lemma}

\begin{proof}
A boundary-exchange $\sigma$ is a Euclidean isometry of the embedded frieze interchanging the two bordering lines $y=0$ and $y=-n$.
Being vertical-reversing, it carries horizontal lines to horizontal lines in reverse order, sending the line $y=-\ell$ to the line $y=-(n-\ell)$ and in particular the bifurcation line $y=-(n-2)$ to the line $y=-2$.
The entries on a horizontal line are collinear and equally spaced, and an isometry preserves the spacing of such a configuration.
The bifurcation line $y=-(n-2)$ carries two entries per period, whereas for $n \ge 5$ the line $y=-2$ carries the single row $x_2$ and thus one entry per period (Subsection~\ref{ssec:Dn}).
An isometry cannot carry a line of spacing $\frac12$ onto a line of spacing $1$, and therefore $\sigma$ cannot send $y=-(n-2)$ to $y=-2$.
This contradiction proves the claim.
\end{proof}

\begin{proposition}\label{prop:D4}
Let $F$ be a frieze pattern of type $D_4$ with period $p$.
\begin{enumerate}
\item The reflection in the bifurcation line is a symmetry of $F$ if and only if $x_{4,j} = x_{1,j+1}$ for all $j$.
\item For $0 < \ell < p$, the glide along the bifurcation line with shift $\ell$ is a symmetry of $F$ if and only if $x_{4,j} = x_{1,j+1-\ell}$ for all $j$ and the rows $x_2$ and $x_3$ are invariant under $j \mapsto j + \ell$.
\end{enumerate}
\end{proposition}

\begin{proof}
For $n = 4$ the bifurcation line is the equator, and a boundary-exchange fixes it as a set and interchanges the lines $y = -1$ and $y = -3$, and therefore the rows $x_1$ and $x_4$.

(1) The reflection $s$ in the equator fixes the bifurcation line pointwise and imposes no condition on $x_2$ and $x_3$.
It carries the entry $x_{1,j}$, at $Q(1,j)$, to $Q(4,j-1)$, and thus $s$ is a symmetry exactly when $x_{1,j} = x_{4,j-1}$ for all $j$, that is $x_{4,j} = x_{1,j+1}$.

(2) The glide is the composite of $s$ with the translation by $\ell$.
The matching of the rows $x_1$ and $x_4$ becomes $x_{4,j} = x_{1,j+1-\ell}$, and the translation of the bifurcation line by $\ell$ requires $x_2$ and $x_3$ to be invariant under $j \mapsto j + \ell$.
\end{proof}

A glide whose shift is a multiple of $p$ is the composite of the reflection with a translation of $F$, and adds no symmetry.

\begin{example}
\label{ex:D4-boundary}
Each kind of boundary-exchange of \textup{Proposition~\ref{prop:D4}} is realised.
The first diagonal $(1,2,3,3)$ generates a frieze of period $4$ with rows $x_1 = (1,3,5,2)$, $x_2 = (2,14,9,1)$ and $x_3 = x_4 = (3,5,2,1)$.
Here $x_{4,j} = x_{1,j+1}$ for all $j$, and the reflection in the bifurcation line is therefore a symmetry, the three leg rows $x_1, x_3, x_4$ carrying one and the same sequence up to cyclic shift in accordance with the triality of $D_4$.
The first diagonal $(1,3,2,1)$ generates a frieze of period $2$ with rows $x_1 = (1,4)$, $x_2 = (3,3)$, $x_3 = (2,2)$ and $x_4 = (1,4)$.
Here $x_{4,j} = x_{1,j}$ with $x_2, x_3$ constant, and the glide of shift $\ell = 1$, half the period, is therefore a symmetry, whereas the mirror is not, since $x_{4,j} \neq x_{1,j+1}$.
In particular type $D_4$ admits boundary-exchange symmetries.
\end{example}

We turn to the vertical reflection.
A vertical reflection preserves each horizontal line, and since it preserves the row $x_1$ its axis lies in $\frac12 \Z$, whence on the bifurcation line it preserves the integral and half-integral positions separately, and thus each of the two interleaved rows.
It amounts to the existence of an integer $c$ with
\[
  x_{i,j} = x_{i,\,c-i-j} \quad (1 \le i \le n-2), \qquad
  x_{n-1,j} = x_{n-1,\,c-(n-1)-j}, \quad
  x_{n,j} = x_{n,\,c-(n-1)-j},
\]
for all $j$.
Hence, exactly as for the types of Section~\ref{sec:main-ABCFG}, the frieze admits a vertical reflection if and only if its rows are simultaneously palindromic about a common vertical axis.
The diagram automorphism exchanging the two legs is not a vertical reflection here, for the legs $x_{n-1}$ and $x_n$ sit on different lines.

\begin{proposition}\label{prop:D-occurrence}
Let $F$ be a frieze pattern of type $D_n$.
\begin{enumerate}
\item If $n \ge 5$, the Conway symbol of $F$ is ${*}\infty\infty$ if $F$ admits a vertical reflection, and $\infty\infty$ otherwise.
\item If $n = 4$, the Conway symbol of $F$ is determined by the presence of a vertical reflection, of the reflection in the equator and of a glide along it whose shift is not a multiple of the period, as in the following table.
\[
\renewcommand{\arraystretch}{1.15}
\begin{array}{ccc|c}
 \text{vertical} & \text{equatorial} & \text{glide} & \text{symbol}\\ \hline
 \text{no}  & \text{no}  & \text{no}  & \infty\infty\\
 \text{no}  & \text{yes} & \text{no}  & \infty{*}\\
 \text{yes} & \text{no}  & \text{no}  & {*}\infty\infty\\
 \text{yes} & \text{no}  & \text{yes} & 2{*}\infty\\
 \text{yes} & \text{yes} & \text{no}  & {*}22\infty
\end{array}
\]
In particular the embedded frieze realises exactly these five symbols.
\end{enumerate}
\end{proposition}

\begin{proof}
Part~(1) follows from Lemma~\ref{lem:D-no-boundary}, which removes boundary-exchange for $n \ge 5$, together with Lemma~\ref{lem:boundary-exch-list} and the fact that the two boundary-exchange-free frieze groups $\infty\infty$ and ${*}\infty\infty$ are separated by the vertical reflection.
For $n=4$ the palindromicity criterion above controls the vertical reflection, and Proposition~\ref{prop:D4} gives the conditions for an equatorial mirror and for a glide.
By Remark~\ref{rmk:7friezes} the symbol is determined by which of a vertical reflection, an equatorial mirror, a glide along the equator and a half-turn the symmetry group contains, and Lemma~\ref{lem:boundary-exch-list} identifies the last three as the boundary-exchange symmetries.
A glide with shift $\ell$ squares to the translation by $2\ell$, whence the period $p$ divides $2\ell$ but not $\ell$ and $\ell$ is congruent to $\frac{p}{2}$ modulo $p$.
Hence an equatorial mirror and a glide are never both symmetries, their composite being a translation by half the period.
The possible combinations are therefore the five rows of the table together with the glide-only group $\infty{\times}$ and the rotation-only group $22\infty$, and all five rows occur.
The first diagonals $(1,1,1,1)$, $(1,2,1,1)$ and $(2,3,2,2)$ give $\infty\infty$, ${*}\infty\infty$ and ${*}22\infty$, while the mirror and glide friezes of Example~\ref{ex:D4-boundary} give $\infty{*}$ and $2{*}\infty$.
The friezes of type $D_4$ are finite in number \cite[Proposition~3.2]{FoP}, and a complete enumeration, recorded in Table~\ref{tab:D4}, shows that no other symbol arises.\footnote{The fifteen patterns of Table~\ref{tab:D4} have periods summing to $51$, the number of friezes of type $D_4$ \cite[Theorem~1.1]{FoP}, and therefore exhaust them.}
\end{proof}

\begin{table}[ht]
\centering
\renewcommand{\arraystretch}{1.2}
\begin{tabular}{lcc}
\hline
Conway symbol & first diagonal & period \\
\hline
$\infty\infty$    & $(1,1,1,1)$ & $4$ \\
                  & $(1,1,2,1)$ & $4$ \\
                  & $(1,3,4,2)$ & $4$ \\
                  & $(1,5,3,3)$ & $4$ \\
\hline
${*}\infty\infty$ & $(1,2,1,1)$ & $4$ \\
                  & $(1,2,3,1)$ & $4$ \\
                  & $(1,3,1,2)$ & $2$ \\
                  & $(2,3,1,4)$ & $2$ \\
\hline
$\infty{*}$       & $(1,1,1,2)$ & $4$ \\
                  & $(1,1,2,2)$ & $4$ \\
                  & $(1,2,3,3)$ & $4$ \\
                  & $(1,3,2,4)$ & $4$ \\
\hline
$2{*}\infty$      & $(1,3,2,1)$ & $2$ \\
\hline
${*}22\infty$     & $(1,2,1,3)$ & $4$ \\
                  & $(2,3,2,2)$ & $1$ \\
\hline
\end{tabular}
\caption{The fifteen friezes of type $D_4$ up to translation.}
\label{tab:D4}
\end{table}

\begin{remark}
\label{rmk:x1-insufficient}
For types $B_n$ and $C_n$ the vertical reflection is detected by the single row $x_1$ \textup{(}Lemma~\ref{lem:x1-determines}\textup{)}, and no such reduction is available for type $D_n$.
The frieze of type $D_8$ with first diagonal $(1,1,2,5,8,3,1,4)$, of period $8$, has $x_1 = (x_{1,1}, \dots, x_{1,8}) = (1,2,3,3,2,1,4,4)$ satisfying $x_{1,\,7-j} = x_{1,j}$ for all $j$, and admits no vertical reflection.
The frieze with first diagonal $(1,1,2,5,8,3,2,2)$ of Remark~\ref{rmk:even-n} has the same rows $x_1, \dots, x_6$ and admits a vertical reflection, and thus in type $D_n$ the row $x_1$ does not even determine the frieze.
\end{remark}

\subsection{A geometric model for type \texorpdfstring{$D_n$}{D\_n}}
\label{ssec:geom-Dn}

It remains to read the vertical reflection off a combinatorial model, as the palindromicity of the quiddity sequence does for type $A_n$ and that of $x_1$ for types $B_n$, $C_n$.
The model is the once-punctured $n$-gon $\mathcal{P}_n$ of Schiffler~\cite{Sch} and Baur--Marsh~\cite{BM}, which carries the cluster category $\mathcal{C}_{D_n} = \mathcal{D}^b(kD_n)/\tau^{-1}[1]$ of Buan, Marsh, Reineke, Reiten and Todorov~\cite{BMRRT}, where $k$ is an algebraically closed field.
The description of $\mathcal{C}_{D_n}$ in terms of $\mathcal{P}_n$ recalled in the next two paragraphs and in Lemma~\ref{lem:Sch} is that of \cite[Sections~3 and 4]{Sch}.

From now on, we label the boundary vertices $0, 1, \dots, n-1$ (modulo $n$) counterclockwise.
The indecomposable objects of $\mathcal{C}_{D_n}$ correspond to the $n^2$ tagged edges of $\mathcal{P}_n$, which are of two kinds.
A \emph{peripheral} edge $M_{a,b}$ runs from the boundary vertex $a$ counterclockwise to the boundary vertex $b$ $(\neq a, a+1)$, the arcs $M_{a,b}$ and $M_{b,a}$ passing on opposite sides of the puncture.
A \emph{puncture} edge $M^{\epsilon}_{a,a}$ joins the boundary vertex $a$ to the puncture, un-notched for $\epsilon = 1$ and notched for $\epsilon = -1$.
Two tagged edges are \emph{compatible} when they can be drawn without crossing in the interior of $\mathcal{P}_n$, two puncture edges $M^{\epsilon}_{a,a}$ and $M^{\eta}_{b,b}$ being compatible precisely when either $a = b$, or $a \neq b$ and $\epsilon = \eta$.

We draw the three kinds of tagged edge for $n = 5$.

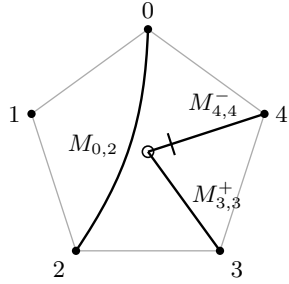
\begin{center}
\begin{tikzpicture}[scale=1.2, font=\small,
  vtx/.style={circle, fill, inner sep=1.1pt},
  punc/.style={circle, draw, inner sep=1.5pt, line width=0.6pt},
  bdry/.style={gray!65, line width=0.5pt},
  edge/.style={line width=0.9pt}]
  \foreach \i/\ang in {0/90, 1/162, 2/234, 3/306, 4/18}
    \coordinate (a\i) at (\ang:1.35);
  \coordinate (o) at (0,0);
  \coordinate (tg) at ($(o)!0.2!(a4)$);
  \draw[bdry] (a0)--(a1)--(a2)--(a3)--(a4)--cycle;
  \draw[edge] (a0) to[bend left=16] node[pos=0.5, left, font=\scriptsize] {$M_{0,2}$} (a2);
  \draw[edge] (a3) -- node[pos=0.55, right, font=\scriptsize] {$M^{+}_{3,3}$} (o);
  \draw[edge] (a4) -- (o);
  \draw[edge, line width=0.8pt] ($(tg)+(108:3.6pt)$) -- ($(tg)+(288:3.6pt)$);
  \node[font=\scriptsize] at (0.72,0.52) {$M^{-}_{4,4}$};
  \foreach \i in {0,1,2,3,4} \node[vtx] at (a\i) {};
  \node[punc] at (o) {};
  \node[above, font=\scriptsize] at (a0) {$0$};
  \node[left, font=\scriptsize] at (a1) {$1$};
  \node[below left, font=\scriptsize] at (a2) {$2$};
  \node[below right, font=\scriptsize] at (a3) {$3$};
  \node[right, font=\scriptsize] at (a4) {$4$};
\end{tikzpicture}
\end{center}

The translation quiver $\Z D_n$, each vertex carrying a Dynkin node $i$ and a column $j$ with $\tau(i,j) = (i,\,j-1)$, is the Auslander--Reiten quiver of $\mathcal{D}^b(kD_n)$ \cite[I.5]{Ha}, the $n^2$ tagged edges forming a fundamental domain for $\mathcal{C}_{D_n}$.
The array of Subsection~\ref{ssec:Dn} realises $\Z D_n$, the horizontal translation being $\tau$ and the rows being the nodes.
Its entries are the values of the cluster variables of type $D_n$ \textup{(}Subsection~\ref{ssec:FP}\textup{)}, and therefore two positions of $\Z D_n$ carrying the same tagged edge carry equal entries and a frieze of type $D_n$ assigns a positive integer $F(\gamma)$ to each tagged edge $\gamma$.
We call $F(\gamma)$ \emph{the value} of $\gamma$. 
The clusters are \emph{the tagged triangulations}, the maximal sets of pairwise compatible tagged edges, which all have cardinality $n$.
Setting the $n$ variables of a tagged triangulation $C$ equal to $1$ gives the unitary frieze $f_C$.

\begin{lemma}[{\cite[Proposition~3.2, Lemma~3.5]{FoP}, and \cite[Proposition~3.1]{GS}}]
\label{lem:FoP}
Every frieze $F$ of type $D_n$ determines a unique tagged triangulation $T_F$ of $\mathcal{P}_n$ with the following two properties.
\begin{enumerate}
\item[(a)] Every peripheral edge $\gamma$ of $T_F$ has $F(\gamma) = 1$, and $T_F$ contains every un-notched edge $\gamma$ with $F(\gamma) = 1$.
\item[(b)] The puncture edges of $T_F$ are either
  \begin{enumerate}
  \item[(b1)] the pair $M^{+1}_{a,a}$ and $M^{-1}_{a,a}$ at a boundary vertex $a$ with
    $F(M^{\pm 1}_{a,a}) = 1$, or
  \item[(b2)] the un-notched edges $M^{+1}_{a,a}$ at $m$ boundary vertices $a$ with
    $F(M^{+1}_{a,a}) = d$ for an integer $d$ dividing $m$, and
    $F(M^{-1}_{a,a}) = \frac{m}{d}$.
  \end{enumerate}
\end{enumerate}
Moreover $F$ is determined by its restriction to $T_F$.
\end{lemma}
We call $(T_F, d)$ \emph{the data} of $F$, taking $d = 1$ in case \textup{(b1)}.
Throughout, $m$ denotes the number of puncture edges of $T_F$, which is $2$ in case \textup{(b1)} and is the $m$ of \textup{(b2)} otherwise.
Conversely, every such pair $(T, d)$, with $T$ a tagged triangulation whose puncture edges are un-notched or form the pair at a single vertex and $d$ a divisor of the number of its puncture edges, the pair admitting $d = 1$ alone, is the data of exactly one frieze, by the count of friezes in \cite[Theorems~1.1 and~3.6]{FoP}, and we shall use only the uniqueness.

\begin{lemma}\label{lem:spokes}
Every tagged triangulation $C$ contains a puncture edge, and its puncture edges either include two with a common tag or consist of the two edges at a single vertex with opposite tags.
\end{lemma}

\begin{proof}
If $C$ contained no puncture edge, the puncture would lie in the interior of a region of $\mathcal{P}_n \setminus C$, and a puncture edge drawn inside that region would be compatible with $C$, contradicting the maximality of $C$.
Puncture edges at distinct vertices are compatible only when their tags coincide, while the two puncture edges at one vertex are compatible.
If the puncture edges of $C$ were the single edge $M^{\epsilon}_{a,a}$, then $C \cup \{M^{-\epsilon}_{a,a}\}$ would again be compatible, contradicting the maximality of $C$.
Hence $C$ has at least two puncture edges.
If two of them lie at distinct vertices, their tags coincide.
Otherwise they all lie at one vertex, where they are the pair of opposite tags.
\end{proof}

We record how the tagged edges sit in $\Z D_n$.
\emph{The tag reversal} is the involution of the set of tagged edges given by $M_{a,b} \mapsto M_{a,b}$ and $M^{\epsilon}_{a,a} \mapsto M^{-\epsilon}_{a,a}$.

\begin{lemma}[{\cite{Sch}}]
\label{lem:Sch}
In the identification of tagged edges with the vertices of $\Z D_n$, an edge $M_{a,b}$ occupies the node $|\delta_{a,b}| - 2$, where $|\delta_{a,b}|$ is the number of boundary vertices on the counterclockwise arc from $a$ to $b$ inclusive, while the $n$ tagged edges with first endpoint $a$ occupy a single column.
Moreover the following hold.
\begin{enumerate}
\item At every column the two leg positions carry the two edges $M^{+1}_{a,a}$ and $M^{-1}_{a,a}$ of a single boundary vertex $a$, and the tag reversal exchanges the nodes $n-1$ and $n$ at each column.
\item The translation acts by $\tau M_{a,b} = M_{a',b'}$ and $\tau M^{\epsilon}_{a,a} = M^{-\epsilon}_{a',a'}$, where $a'$ is the clockwise neighbour of $a$.
\item $\tau^{n}$ is the tag reversal for odd $n$ and the identity for even $n$.
\end{enumerate}
In particular the tags carried by each leg node alternate from column to column.
\end{lemma}

\begin{example}
\label{ex:Dn-frieze}
We locate a tagged triangulation in the array.
The figure shows the tagged triangulation $C$ of $\mathcal{P}_5$ with edges $M_{0,2}$, $M_{0,3}$, $M_{0,4}$, $M^{+}_{0,0}$ and $M^{-}_{0,0}$.
By Lemma~\ref{lem:Sch} the three peripheral edges occupy the nodes $1$, $2$ and $3$ and the two puncture edges the legs $4$ and $5$.
All five share the first endpoint $0$ and therefore lie in a single column.
Setting $C$ equal to $1$ amounts to setting a first diagonal equal to $1$, and $f_{C}$ is the frieze of Example~\ref{ex:D5-embed}, of period $5$, the five edges of $C$ being the entries equal to $1$ on the leftmost diagonal of its picture, one on each of the rows $x_1, \dots, x_5$.
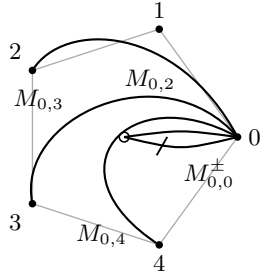
\begin{center}
\begin{tikzpicture}[scale=1.0, baseline=(current bounding box.center), font=\small,
  vtx/.style={circle, fill, inner sep=1.0pt},
  punc/.style={circle, draw, inner sep=1.3pt, line width=0.6pt},
  bdry/.style={gray!65, line width=0.5pt}, edge/.style={line width=0.8pt}]
  \foreach \i/\ang in {0/0, 1/72, 2/144, 3/216, 4/288} \coordinate (v\i) at (\ang:1.5);
  \coordinate (p) at (0,0);
  \draw[bdry] (v0)--(v1)--(v2)--(v3)--(v4)--cycle;
  \draw[edge] (v0) ..controls (0.7,1.5) and (-0.8,1.55).. (v2);
  \draw[edge] (v0) ..controls (0.5,1.15) and (-1.45,0.25).. (v3);
  \draw[edge] (v0) ..controls (0.05,0.8) and (-1.05,-0.45).. (v4);
  \draw[edge] (v0) ..controls (0.75,0.1).. (p);
  \draw[edge] (v0) ..controls (0.75,-0.18).. (p);
  \draw[edge,line width=0.7pt] ($(0.5,-0.12)+(60:4pt)$)--($(0.5,-0.12)+(240:4pt)$);
  \node[font=\scriptsize] at (0.34,0.74){$M_{0,2}$};
  \node[font=\scriptsize] at (-1.13,0.48){$M_{0,3}$};
  \node[font=\scriptsize] at (-0.30,-1.31){$M_{0,4}$};
  \node[font=\scriptsize] at (1.12,-0.5){$M^{\pm}_{0,0}$};
  \foreach \i in {0,1,2,3,4}\node[vtx] at (v\i){};
  \node[punc] at (p){};
  \node[right,font=\scriptsize] at (v0){$0$};
  \node[above,font=\scriptsize] at (v1){$1$};
  \node[above left,font=\scriptsize] at (v2){$2$};
  \node[below left,font=\scriptsize] at (v3){$3$};
  \node[below,font=\scriptsize] at (v4){$4$};
\end{tikzpicture}
\end{center}
Here $x_1 = (1,2,2,2,7)$ is not palindromic, $f_{C}$ admits no vertical reflection and its Conway symbol is $\infty\infty$ by Proposition~\ref{prop:D-occurrence}\,(1).
\end{example}

\begin{lemma}\label{lem:value-one}
Let $F$ be a frieze of type $D_n$ and $b$ a boundary vertex with $F(M^{\epsilon}_{b,b}) = 1$ for some tag $\epsilon$.
Then $b$ carries a puncture edge of $T_F$.
\end{lemma}

\begin{proof}
We note that two tagged edges of value $1$ are compatible \cite[Lemma~3.3]{FoP}.
We argue by contraposition and let $b$ carry no puncture edge of $T_F$, so that neither $M^{+1}_{b,b}$ nor $M^{-1}_{b,b}$ lies in $T_F$.

In case \textup{(b1)}, the pair lying at a vertex $a$, an edge $M^{\epsilon}_{b,b}$ of value $1$ would be compatible with $M^{-\epsilon}_{a,a}$, also of value $1$, against their opposite tags at distinct vertices, a contradiction.

In case \textup{(b2)}, the edge $M^{+1}_{b,b}$ is compatible with all the puncture edges of $T_F$ and lies outside $T_F$, whence by the maximality of $T_F$ it is incompatible with a peripheral edge $\gamma$ of $T_F$.
Now, $\gamma$ has value $1$ by Lemma~\ref{lem:FoP}\,\textup{(a)}, and therefore $F(M^{+1}_{b,b}) > 1$.
The tag plays no part in the compatibility of a puncture edge with a peripheral one, and thus $M^{-1}_{b,b}$ is incompatible with $\gamma$ as well and $F(M^{-1}_{b,b}) > 1$.
\end{proof}

The dihedral group $\mathrm{Dih}_n = \langle\, \rho,\, s \mid \rho^{n} = s^{2} = 1,\ s\rho s = \rho^{-1} \,\rangle$ of $\mathcal{P}_n$ acts on tagged edges as follows.
\[
  \rho M_{a,b} = M_{a+1,b+1}, \quad
  \rho M^{\epsilon}_{a,a} = M^{\epsilon}_{a+1,a+1},\quad 
  s M_{a,b} = M_{s(b),s(a)}, \quad
  s M^{\epsilon}_{a,a} = M^{\epsilon}_{s(a),s(a)}
\]
The image of a set $T$ of tagged edges is written $sT = \{\, s\gamma \mid \gamma \in T \,\}$.

\begin{lemma}\label{lem:D-equivariance}
Let $F$ be a frieze of type $D_n$ with data $(T_F, d)$, and let $s$ be a reflection of $\mathcal{P}_n$.
\begin{enumerate}
\item On the data the reflection $s$ acts by $(T_F,\, d) \mapsto (sT_F,\, d)$.
The tag reversal acts as the identity in case \textup{(b1)} and by $(T_F,\, d) \mapsto (T_F,\, \frac{m}{d})$ in case \textup{(b2)}.
\item $F$ is invariant under $s$ if and only if $sT_F = T_F$.
\item $F$ is invariant under the composite of $s$ with the tag reversal if and only if $sT_F = T_F$ and either case \textup{(b1)} holds or $d^{2} = m$.
\end{enumerate}
\end{lemma}

\begin{proof}
(1) A reflection carries compatible edges to compatible ones.
This implies that it takes a triangulation with the properties of Lemma~\ref{lem:FoP} for $F$ to one with those properties for $F \circ s^{-1}$.
The values of the puncture edges are unchanged, and the uniqueness in Lemma~\ref{lem:FoP} gives $(sT_F,\, d)$ as the data of $F \circ s^{-1}$.

Write $t$ for the tag reversal.
We claim that $T_F$ has the properties of Lemma~\ref{lem:FoP} for $F \circ t$ as well.
Since the tag reversal fixes each peripheral edge $\gamma$ of $T_F$, 
we have
\[ (F \circ t)(\gamma) = F(t\gamma) = F(\gamma) = 1. \]
Condition \textup{(b)} holds for $F \circ t$ in the same case as for $F$.
In case \textup{(b1)} the map $t$ exchanges the two puncture edges at the vertex, both of value $1$.
In case \textup{(b2)} it sends each un-notched puncture edge to its notched partner, of value $\frac{m}{d}$, a divisor of $m$.
There remains the requirement that $T_F$ contain every un-notched edge of value $1$ for $F \circ t$.
The peripheral ones among these have the same value for $F$ and lie in $T_F$ already, and the others are the edges $M^{+1}_{b,b}$ with $F(M^{-1}_{b,b}) = 1$, which by Lemma~\ref{lem:value-one} occur only at vertices carrying a puncture edge of $T_F$ and therefore lie in $T_F$.
Uniqueness gives the data $(T_F,\, 1)$ in case \textup{(b1)} and $(T_F,\, \frac{m}{d})$ in case \textup{(b2)}.

(2) The frieze $F$ is determined by its data.
By \textup{(1)} the data of $F \circ s^{-1}$ is $(sT_F,\, d)$, which is the data of $F$ exactly when $sT_F = T_F$.

(3) The maps $s$ and $t$ commute, $s \circ t$ and $t \circ s$ both sending $M_{a,b}$ to $M_{s(b),s(a)}$ and $M^{\epsilon}_{a,a}$ to $M^{-\epsilon}_{s(a),s(a)}$.
By \textup{(1)} the composite carries $F$ to the frieze with data $(sT_F,\, d)$ in case \textup{(b1)} and $(sT_F,\, \frac{m}{d})$ in case \textup{(b2)}, the case and the number $m$ being unchanged by $s$.
This is the data of $F$ exactly when $sT_F = T_F$ and either case \textup{(b1)} holds or $d^{2} = m$.
\end{proof}

To turn this into a criterion we identify the geometric symmetries of $\mathcal{P}_n$ with transformations of the embedded frieze.

\begin{lemma}\label{lem:G1}
Each rotation and each reflection of $\mathcal{P}_n$ induces a cluster \textup{(}anti\textup{)}automorphism of $\mathcal{C}_{D_n}$ in the sense of \cite{CZ}, rotations acting covariantly on $\Z D_n$ and reflections contravariantly.
\end{lemma}

\begin{proof}
The crossing number of two tagged edges equals the $k$-dimension of $\mathrm{Ext}^1$ between the corresponding objects of $\mathcal{C}_{D_n}$ \cite[Theorem~5.3]{Sch}, while the arrows of $\Z D_n$ are the irreducible morphisms.
A homeomorphism of $\mathcal{P}_n$ preserves crossing numbers.
This implies that it preserves compatibility and carries clusters to clusters.
A rotation is orientation-preserving, and therefore carries arrows to arrows and acts covariantly on $\Z D_n$.
A reflection is orientation-reversing, and thus reverses every arrow and acts contravariantly.
\end{proof}

\begin{lemma}\label{lem:G2}
Let $n\ge5$ and let $s$ be a reflection of $\mathcal{P}_n$.
\begin{enumerate}
\item If the axis of $s$ passes through a boundary vertex, then $s$ fixes every node and is, under the embedding, a vertical reflection of the embedded frieze.
\item If $n$ is even and the axis of $s$ joins two edge midpoints, then $s$ transposes the two legs and is therefore not a vertical reflection, whereas the composite of $s$ with the tag reversal fixes every node and is a vertical reflection.
\end{enumerate}
For odd $n$ every axis passes through a vertex, and case \textup{(1)} applies to all reflections.
\end{lemma}

\begin{proof}
By Lemma~\ref{lem:G1} the reflection $s$ acts contravariantly on $\Z D_n$.
Thus $s$ conjugates $\tau$ into $\tau^{-1}$ and permutes the nodes by a diagram automorphism $\pi$.
For $n \ge 5$ the diagram $D_n$ has only the identity and the transposition of the two legs as automorphisms.
Write the action of $s$ on $\Z D_n$ as $(i,j) \mapsto (\pi(i),\, f_i(j))$.
Since $\tau(i,j) = (i,\, j-1)$ and $s\tau s^{-1} = \tau^{-1}$, we have
\[
  \bigl(\pi(i),\, f_i(j-1)\bigr) = s\tau(i,j) = \tau^{-1}s(i,j)
  = \bigl(\pi(i),\, f_i(j)+1\bigr),
\]
whence $f_i(j-1) = f_i(j) + 1$.
Hence $f_i(j) + j$ does not depend on $j$, and we write $C_i$ for its value.
By the relations~\eqref{eq:Dn-relations} an arrow raising the node preserves the column, and an arrow lowering it raises the column by one.
Since $s$ is contravariant, the image under $s$ of the arrow $(i-1,j) \to (i,j)$ is the arrow $(\pi(i),\, f_i(j)) \to (\pi(i-1),\, f_{i-1}(j))$, and as $\pi(i-1) < \pi(i)$ we have $f_{i-1}(j) = f_i(j) + 1$, that is $C_{i-1} = C_i + 1$ for $2 \le i \le n-2$.
The arrows $(n-2,j) \to (n-1,j)$ and $(n-2,j) \to (n,j)$ give $C_{n-2} = C_{n-1} + 1 = C_n + 1$.
With $c = C_1 + 1$ the action of $s$ is therefore
\begin{equation}\label{eq:s-action}
  s \colon (i,j) \longmapsto
  \begin{cases}
    (\pi(i),\, c-i-j) & \text{if $1 \le i \le n-2$}, \\
    (\pi(i),\, c-(n-1)-j) & \text{if $i = n-1,\, n$}.
  \end{cases}
\end{equation}
A direct computation gives
\[
  Q(i,j) + Q\bigl(s(i,j)\bigr) =
  \begin{cases}
    \bigl(c,\ -2i\bigr) & (1 \le i \le n-2), \\
    \bigl(c,\ -2(n-2)\bigr) & (i = n-1,\ \pi = \mathrm{id}), \\
    \bigl(c,\ -2(n-1)\bigr) & (i = n,\ \pi = \mathrm{id}), \\
    \bigl(c,\ -(2n-3)\bigr) & (i = n-1,\, n,\ \pi \neq \mathrm{id}).
  \end{cases}
\]
The first coordinate is $c$ in every case and does not depend on $\pi$.
The second coordinates of $Q(i,j)$ and $Q(s(i,j))$ coincide exactly when $\pi = \mathrm{id}$.
For $\pi = \mathrm{id}$ the map $Q(i,j) \mapsto Q(s(i,j))$ is $(x,y) \mapsto (c-x,\, y)$, the reflection in the line $x = \frac{c}{2}$.
For the transposition of the legs the two points lie on the distinct lines $y = -(n-2)$ and $y = -(n-1)$, and this map is not a vertical reflection.
It suffices to determine $\pi$. 

(1) Suppose the axis passes through the vertex $a_0$.
As a reflection preserves the tag, $s$ fixes each of the two edges $M^{\pm1}_{a_0,a_0}$.
These occupy the two leg positions of a single column by Lemma~\ref{lem:Sch}\,\textup{(1)}, and $\pi$ fixes each leg and $\pi = \mathrm{id}$.

(2) Suppose $n$ is even and the axis joins two edge midpoints.
Then $s$ exchanges two adjacent boundary vertices, whose puncture edges occupy two adjacent columns $j_0$ and $j_0-1$ by Lemma~\ref{lem:Sch}\,\textup{(1)} and \textup{(2)}.
The action~\eqref{eq:s-action} on the legs gives 
\[ c-(n-1)-j_0 = j_0-1, \] 
that is, $c = 2j_0+n-2$.
As $n$ is even, $c$ is even and $c-(n-1)$ is odd.
The parity of the column of every puncture edge is reversed.
It follows from Lemma~\ref{lem:Sch} that a puncture edge of tag $\epsilon$ lies on the leg node determined by $\epsilon$ and the parity of its column.
Its image carries the same tag and the opposite column parity, and lies on the other leg.
Therefore $\pi$ is the transposition of the legs.
The tag reversal exchanges the leg positions at every column by Lemma~\ref{lem:Sch}\,\textup{(1)}, and the composite of $s$ with it has $\pi = \mathrm{id}$ and is a vertical reflection.
\end{proof}

\begin{lemma}\label{lem:G3}
For $n\ge5$ the cluster anti-automorphisms of $\mathcal{C}_{D_n}$ are the $n$ reflections of $\mathcal{P}_n$, each taken with or without the tag reversal, and every vertical reflection of the embedded frieze is induced by one of these $2n$ elements.
\end{lemma}
\begin{proof}
A cluster automorphism in the sense of \cite{CZ} is permitted to reverse the quiver, and the quiver-reversing ones are precisely the contravariant anti-automorphisms of Lemma~\ref{lem:G1}.
By the classification of Chang--Zhu the group of cluster \textup{(}anti\textup{)}automorphisms of type $D_n$ is $\mathrm{Dih}_n\times\Z_2$ of order $4n$ for $n\ge5$ \cite[Table~2]{CZ}.

The rotations and reflections of $\mathcal{P}_n$ together with the tag reversal, the latter induced by the auto-equivalence exchanging the two leg nodes \cite[Proposition~4.1\,(4)]{Sch}, realise through Lemma~\ref{lem:G1} a subgroup of order $4n$.
Indeed a rotation or a reflection preserves the tag of every edge while the tag reversal fixes every peripheral edge and moves every puncture edge.
This implies that an element acting trivially on tagged edges has trivial tag component, and a non-trivial rotation carries $M_{a,a+2}$ to a different edge while a reflection fixing every peripheral edge would satisfy $s(a) = b$ for every pair $a$, $b$ with $b \neq a, a+1$, which is impossible for $n \ge 4$.
Hence they exhaust the group, and the anti-automorphisms are the $2n$ elements.
A vertical reflection of the embedded frieze reverses every arrow of $\Z D_n$, and is therefore one of the $2n$ anti-automorphisms.
It fixes every node and conjugates $\tau$ into $\tau^{-1}$.
By Lemma~\ref{lem:G2} the elements acting in this way are the reflections through a vertex and, for even $n$, the edge-axis reflections composed with the tag reversal.
A vertex-axis reflection composed with the tag reversal exchanges the two legs by Lemma~\ref{lem:Sch}\,\textup{(1)}, and does not act in this way.
Hence every vertical reflection of the embedded frieze arises from one of the $2n$ elements.
\end{proof}

\begin{proposition}\label{prop:D-geometric}
Let $n \ge 5$ and let $F$ be a frieze pattern of type $D_n$, with data $(T_F, d)$ and $m$ puncture edges.
Then $F$ admits a vertical reflection if and only if one of the following holds.
\begin{enumerate}
\item $T_F$ is invariant under a reflection of $\mathcal{P}_n$ whose axis passes through a boundary vertex.
\item $n$ is even, $d^{2} = m$, and $T_F$ is invariant under a reflection whose axis joins two edge midpoints.
\end{enumerate}
The Conway symbol of $F$ is ${*}\infty\infty$ in that case and $\infty\infty$ otherwise.
For odd $n$ the second condition is vacuous, and the first is the bilateral symmetry of $T_F$.
\end{proposition}

\begin{proof}
By Proposition~\ref{prop:D-occurrence}\,(1) it suffices to identify the existence of a vertical reflection with the stated condition.
By Lemma~\ref{lem:G3} a vertical reflection of $F$ is induced by a reflection $s$ of $\mathcal{P}_n$ taken with or without the tag reversal, and by Lemma~\ref{lem:G2} the elements acting node-preservingly are the reflections through a vertex and, for even $n$, the composites of an edge-axis reflection with the tag reversal.

(1) Let $s$ be a reflection through a vertex.
Lemma~\ref{lem:D-equivariance}\,(2) gives $sT_F = T_F$, and conversely such an $s$ leaves $F$ invariant and acts as a vertical reflection by Lemma~\ref{lem:G2}\,(1).

(2) Let $s$ be an edge-axis reflection, and consider its composite with the tag reversal.
Lemma~\ref{lem:D-equivariance}\,(3) gives $sT_F = T_F$ together with one of two alternatives.
The first is impossible, an edge axis fixing no boundary vertex while $sT_F = T_F$ would carry the pair at a vertex $a$ to the pair at $s(a)$ and force $s(a) = a$.
Hence $d^{2} = m$, and conversely Lemma~\ref{lem:G2}\,(2) makes the composite a vertical reflection.
\end{proof}

\begin{example}
\label{ex:D5-spokes}
Let $C_0$ be the tagged triangulation of $\mathcal{P}_5$ consisting of the five edges $M^{+}_{a,a}$ $(0 \le a \le 4)$, drawn on the left together with the axis through the boundary vertex $0$.
Every axis of $\mathcal{P}_5$ passes through a boundary vertex and $C_0$ is invariant under each of the five reflections, and the Conway symbol of $f_{C_0}$ is ${*}\infty\infty$ by Proposition~\ref{prop:D-geometric}.
The frieze is displayed on the right, of period $2$, its rows being
\[
  x_1 = (2,2), \quad x_2 = (3,3), \quad x_3 = (4,4), \quad
  x_4 = (1,5), \quad x_5 = (5,1).
\]
In the five columns displayed, which carry the $25$ tagged edges of $\mathcal{P}_5$, the entries equal to $1$ are the five ringed ones, three on the row $x_4$ and two on the row $x_5$, in accordance with the alternation of tags of Lemma~\ref{lem:Sch}.
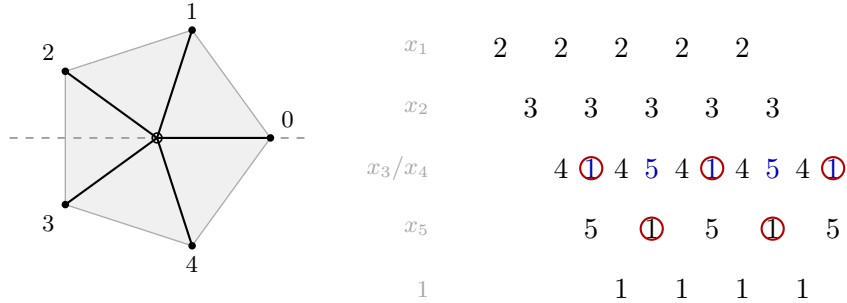
\begin{center}
\begin{tikzpicture}[scale=1.0, baseline=(current bounding box.center), font=\small,
  vtx/.style={circle, fill, inner sep=1.0pt},
  punc/.style={circle, draw, inner sep=1.3pt, line width=0.6pt},
  bdry/.style={gray!65, line width=0.5pt}, edge/.style={line width=0.8pt}]
  \foreach \i/\ang in {0/0, 1/72, 2/144, 3/216, 4/288} \coordinate (v\i) at (\ang:1.5);
  \coordinate (p) at (0,0);
  \foreach \i/\j in {0/1, 1/2, 2/3, 3/4, 4/0}
    \fill[gray!12] (p) -- (v\i) -- (v\j) -- cycle;
  \draw[dashed, gray!75, line width=0.7pt] (-1.95,0) -- (1.95,0);
  \draw[bdry] (v0)--(v1)--(v2)--(v3)--(v4)--cycle;
  \foreach \i in {0,1,2,3,4} \draw[edge] (v\i) -- (p);
  \foreach \i in {0,1,2,3,4} \node[vtx] at (v\i){};
  \node[punc] at (p){};
  \node[above right,font=\scriptsize] at (v0){$0$};
  \node[above,font=\scriptsize] at (v1){$1$};
  \node[above left,font=\scriptsize] at (v2){$2$};
  \node[below left,font=\scriptsize] at (v3){$3$};
  \node[below,font=\scriptsize] at (v4){$4$};
\end{tikzpicture}
\hspace{0.4cm}
\begin{tikzpicture}[scale=0.8, baseline=(current bounding box.center), font=\small,
  num/.style={inner sep=1.3pt},
  numt/.style={inner sep=1.3pt, text=blue!70!black},
  ring/.style={draw=red!70!black, line width=0.8pt}]
  \foreach \x in {0,1,2,3,4} \node[num] at (\x,0) {$1$};
  \foreach \x/\v in {0.5/2, 1.5/2, 2.5/2, 3.5/2, 4.5/2} \node[num] at (\x,-1) {$\v$};
  \foreach \x/\v in {1/3, 2/3, 3/3, 4/3, 5/3} \node[num] at (\x,-2) {$\v$};
  \foreach \x/\v in {1.5/4, 2.5/4, 3.5/4, 4.5/4, 5.5/4} \node[num] at (\x,-3) {$\v$};
  \foreach \x/\v in {2/1, 3/5, 4/1, 5/5, 6/1} \node[numt] at (\x,-3) {$\v$};
  \foreach \x/\v in {2/5, 3/1, 4/5, 5/1, 6/5} \node[num] at (\x,-4) {$\v$};
  \foreach \x in {2.5,3.5,4.5,5.5} \node[num] at (\x,-5) {$1$};
  \foreach \x/\y in {2/-3, 4/-3, 6/-3, 3/-4, 5/-4}
    \draw[ring] (\x,\y) circle (5.2pt);
  \foreach \r/\lab in {0/1, -1/x_1, -2/x_2, -3/{x_3/x_4}, -4/x_5, -5/1}
    \node[anchor=east, gray!70, font=\scriptsize] at (-0.5,\r) {$\lab$};
\end{tikzpicture}
\end{center}
Together with Example~\ref{ex:Dn-frieze} this exhibits the two cases of Lemma~\ref{lem:spokes}, the triangulation $C$ carrying the pair of puncture edges of opposite tags at one vertex and $C_0$ carrying puncture edges of a common tag at distinct vertices.
\end{example}

\begin{proposition}\label{prop:unitary}
A frieze $F$ of type $D_n$ is unitary if and only if $d = 1$ or $d = m$.
\end{proposition}

\begin{proof}
A frieze with $d = 1$ takes the value $1$ on $T_F$, and one with $d = m$ takes it on the peripheral part of $T_F$ together with the notched edges at the same vertices, which is again a tagged triangulation.
Conversely the edges of value $1$ of a unitary frieze are the variables of the tagged triangulation on which it takes that value \cite[Proposition~2.5]{GS}.
Its puncture edges lie at the vertices carrying those of $T_F$ by Lemma~\ref{lem:value-one}, and by Lemma~\ref{lem:spokes} they either carry a common tag or are the pair at a single vertex.
When they carry a common tag they are the un-notched puncture edges of $T_F$ or the notched edges at the same vertices, which gives $d = 1$ or $\frac{m}{d} = 1$, and when they are the pair at a single vertex that vertex carries an un-notched puncture edge of $T_F$ of value $1$, which gives $d = 1$ again.
\end{proof}

\begin{remark}
\label{rmk:even-n}
For even $n$ the condition of Proposition~\ref{prop:D-geometric} can be strictly finer than bilateral symmetry, the value $d$ entering.
Write the boundary vertices of $\mathcal{P}_8$ as $0, \dots, 7$ and let $T$ be the tagged triangulation
\[
  \{\, M_{0,2},\, M_{0,3},\, M_{4,7},\, M_{5,7},\,
       M^{+1}_{0,0},\, M^{+1}_{3,3},\, M^{+1}_{4,4},\, M^{+1}_{7,7} \,\},
\]
which is invariant under the reflection through the midpoints of $\{7,0\}$ and $\{3,4\}$ but under no reflection through a vertex.
It has $m = 4$ puncture edges, and each of $d = 1, 2, 4$ is the data of a frieze.
For $d = 1$ and $d = 4$ these are the unitary frieze $f_T$ and its tag reversal, with first diagonals $(1,1,2,5,8,3,1,4)$ and $(1,1,2,5,8,3,4,1)$, the former being the frieze of Remark~\ref{rmk:x1-insufficient}, and both have symbol $\infty\infty$ by Proposition~\ref{prop:D-geometric}.
For $d = 2$ it is the frieze with first diagonal $(1,1,2,5,8,3,2,2)$, of period $8$, which shares the rows $x_1, \dots, x_6$ with the other two and differs from them in the leg rows alone.
Here $d^{2} = m$, condition~\textup{(2)} holds, and the symbol is ${*}\infty\infty$.
Thus one and the same triangulation carries friezes with and without a vertical reflection, the value $d$ deciding between them.
Condition~\textup{(2)} occurs only for non-unitary friezes by Proposition~\ref{prop:unitary}, the equality $d^{2} = m$ forcing $d$ to be neither $1$ nor $m$.
For $D_6$ every reflection-symmetric tagged triangulation of $\mathcal{P}_6$ is also vertex-axis-symmetric, as a direct enumeration shows, and for $D_8$ the rotations of $T$ are the only tagged triangulations of the form of Lemma~\ref{lem:FoP} that are invariant under an edge-axis reflection but under no vertex-axis reflection and whose number of puncture edges is a perfect square.
Hence for $5 \le n \le 8$ the frieze with $d = 2$ above is, up to translation, the only one for which condition~\textup{(2)} is needed.
\end{remark}

The hypothesis $n\ge5$ enters through Lemma~\ref{lem:G2}.
For $D_4$ the diagram automorphism group is the symmetric group $S_3$---the triality---and the cluster automorphism group is the larger $\mathrm{Dih}_4\times S_3$ of order $48$ \cite[Table~2]{CZ}.
Its extra anti-automorphisms permute the three legs of $D_4$ across the nodes of the embedded frieze, and these are precisely the boundary-exchanges of Proposition~\ref{prop:D4}.
They are responsible for the three further symbols $\infty{*}$, $2{*}\infty$ and ${*}22\infty$ that $D_4$ realises beyond the dichotomy of Proposition~\ref{prop:D-geometric}.

We close with the folding of type $D_{n+1}$ onto type $C_n$, which places Theorem~\ref{thm:main-BCFG} and Proposition~\ref{prop:C3-classification} within the present framework.
Let $n \ge 3$.
When the two leg rows of a frieze $F = (x_{i,j})$ of type $D_{n+1}$ coincide, $x_{n,j} = x_{n+1,j}$ for all $j$, deleting the row $x_n$ leaves the array with rows $x_1, \dots, x_{n-1}, x_{n+1}$, which we call \emph{the fold} of $F$.

\begin{proposition}\label{prop:fold}
Let $n \ge 3$.
\begin{enumerate}
\item The fold of a frieze $F$ of type $D_{n+1}$ with coinciding leg rows is a frieze of type $C_n$, and every frieze of type $C_n$ is the fold of exactly one such $F$.
\item A frieze $F$ of type $D_{n+1}$ has coinciding leg rows if and only if its data $(T_F, d)$ falls under case \textup{(b1)} or satisfies $d^{2} = m$. In the former case $F$ is unitary, and in the latter it is not.
\item The Conway symbols of $F$ and of its fold coincide.
\end{enumerate}
\end{proposition}

\begin{proof}
(1) With $x_n = x_{n+1}$ the relation~\eqref{eq:Dn-relations} at the trivalent node $n-1$ reads $x_{n-1,j}\, x_{n-1,j+1} = 1 + x_{n-2,j+1}\, x_{n+1,j}^{\,2}$ and the relation at the leg $n+1$ reads $x_{n+1,j}\, x_{n+1,j+1} = 1 + x_{n-1,j+1}$, while the remaining relations do not involve the legs.
These are the relations of type $C_n$ of Subsection~\ref{sssec:Cn} for the rows $x_1, \dots, x_{n-1}, x_{n+1}$, the last one in the role of $x_n$.
Conversely, a frieze of type $C_n$ with its last row repeated as the row $x_n$ satisfies~\eqref{eq:Dn-relations}, and the two constructions are mutually inverse.
The fold is embedded as in Definition~\ref{df:basic-domain}, the row $x_{n+1}$ of $F$ occupying the positions $P(n,j)$.

(2) The tag reversal exchanges the two leg positions in every column and fixes the others \textup{(}Lemma~\ref{lem:Sch}\,\textup{(1)}\textup{)}, whence the leg rows of $F$ coincide exactly when $F \circ t = F$.
By Lemma~\ref{lem:D-equivariance}\,\textup{(1)} the data of $F \circ t$ is $(T_F, 1)$ in case \textup{(b1)} and $(T_F, \frac{m}{d})$ in case \textup{(b2)}, and since the data determines the frieze, $F \circ t = F$ holds exactly in case \textup{(b1)} or when $d^{2} = m$.
In case \textup{(b1)} the frieze is unitary by Proposition~\ref{prop:unitary}, and $d^{2} = m$ with $m \ge 2$ excludes $d = 1$ and $d = m$.

(3) The rows $x_1$ and, for $n = 3$, $x_4$ have half-integral abscissae, whence every symmetry of $F$ acts on abscissae by $x \mapsto x + \ell$ or $x \mapsto c - x$ with $\ell, c \in \Z$ and preserves the parity of $2x$, which distinguishes the two rows interleaved on the bifurcation line.
Every symmetry of $F$ therefore carries the positions of the row $x_n$ onto themselves and restricts to a symmetry of the fold.
Conversely, a vertical reflection of the fold extends to $F$, the rows $x_n$ and $x_{n+1}$ carrying equal values at equal abscissae, and the reflection in the equator of the fold, which occurs only for $n = 3$, extends to $F$ because it fixes the bifurcation line pointwise.
For $n \ge 4$ the symbol is decided on both sides by the vertical reflection \textup{(}Theorem~\ref{thm:main-BCFG} and Proposition~\ref{prop:D-occurrence}\,\textup{(1)}\textup{)}.
For $n = 3$ the symbol of the fold is decided by the vertical and the equatorial reflection \textup{(}Proposition~\ref{prop:C3-classification}\textup{)}, no frieze of type $C_3$ admitting a glide, whence $F$ admits no glide either and its symbol is decided by the same two reflections \textup{(}Proposition~\ref{prop:D-occurrence}\,\textup{(2)}\textup{)}.
\end{proof}

\begin{corollary}\label{cor:Cn-geometric}
Let $n \ge 4$, let $G$ be a frieze pattern of type $C_n$, and write $(T_F, d)$ for the data of the frieze $F$ of type $D_{n+1}$ whose fold is $G$.
Then the Conway symbol of $G$ is ${*}\infty\infty$ if $T_F$ is invariant under a reflection of $\mathcal{P}_{n+1}$ through a boundary vertex, or if $d^{2} = m$ and $T_F$ is invariant under some reflection of $\mathcal{P}_{n+1}$, and $\infty\infty$ otherwise.
\end{corollary}

\begin{proof}
By Proposition~\ref{prop:fold}\,\textup{(3)} the symbol of $G$ is that of $F$, which Proposition~\ref{prop:D-geometric} reads off the data.
In case \textup{(b1)} we have $d = 1$ and $m = 2$, and only condition \textup{(1)} of Proposition~\ref{prop:D-geometric} can hold, while for $d^{2} = m$ its two conditions together amount to the invariance of $T_F$ under some reflection.
\end{proof}

Through the fold, the triality of $D_4$ is the source of the exceptional type $C_3$ as well, Proposition~\ref{prop:C3-classification} being Proposition~\ref{prop:D-occurrence}\,\textup{(2)} restricted to the friezes of Table~\ref{tab:D4} with coinciding leg rows.

Table~\ref{tab:summary} collects the classification.
In the column of type $A_n$, B and S mean that the triangulation is bilateral and centrally symmetric \textup{(}Theorem~\ref{thm:main-An}\textup{)}.
In the other columns, P means that the first interior row $x_1$ is palindromic \textup{(}Theorem~\ref{thm:main-BCFG} and Proposition~\ref{prop:C3-classification}\textup{)}, E that the reflection in the equator is a symmetry, which for $C_3$ reads $x_{3,j} = x_{1,j+1}$ and for $D_4$ reads $x_{4,j} = x_{1,j+1}$, V that a vertical reflection is a symmetry, which for $D_n$ is read off the data by Proposition~\ref{prop:D-geometric}, and G that a glide along the equator is a symmetry \textup{(}Proposition~\ref{prop:D4}\textup{)}.
For $D_4$ the word only indicates that the remaining two features are absent.

\begin{table}[ht]
\centering
\renewcommand{\arraystretch}{1.2}
\begin{tabular}{lccccc}
\hline
Symbol & $A_n$ & \begin{tabular}[c]{@{}c@{}} $B_n$, $C_n$ $(n \ne 3)$ \\ $F_4$, $G_2$ \end{tabular} & $C_3$ & $D_n$ $(n \ge 5)$ & $D_4$ \\
\hline
${*}22\infty$      & B and S           & --    & P and E           & --    & V and E \\
$2{*}\infty$       & B, not S          & --    & --                & --    & V and G \\
$22\infty$         & --                & --    & --                & --    & -- \\
${*}\infty\infty$  & --                & P     & P, not E          & V     & V only \\
$\infty{*}$        & S, not B          & --    & E, not P          & --    & E only \\
$\infty{\times}$   & not B, not S      & --    & --                & --    & -- \\
$\infty\infty$     & --                & not P & not P, not E      & not V & none \\
\hline
\end{tabular}
\caption{Occurrence of the seven Conway symbols. }
\label{tab:summary}
\end{table}


\end{document}